\documentclass[11pt,english,letterpaper]{article}
\date{December 27, 2023}
\usepackage[final]{microtype}
\usepackage{amsmath}
\usepackage{amsthm}
\usepackage{libertineRoman}
\usepackage{fmtcount}
\usepackage{refcount}
\usepackage[libertine]{newtxmath}
\DeclareMathAlphabet{\mathbb}{U}{msb}{m}{n}
\usepackage[varqu,varl]{inconsolata}
\usepackage[T1]{fontenc}
\usepackage[utf8]{inputenc}
\usepackage{geometry}
\usepackage{color}
\usepackage[english]{babel}
\usepackage{textcomp}
\usepackage{mathrsfs}
\usepackage{mathtools}
\usepackage{url}
\usepackage{enumitem}
\usepackage{tikz}
\usepackage{pgfplots}
\pgfplotsset{compat=1.18}
\usepackage[raggedright]{titlesec}
\usepackage{mleftright}
\mleftright
\usepackage{xcolor}

\usepackage{tikz}

\usepackage[unicode=true,pdfusetitle,
 bookmarks=true,bookmarksnumbered=false,bookmarksopen=false,
 breaklinks=true,pdfborder={0 0 0},pdfborderstyle={},backref=false,colorlinks=true]
 {hyperref}
\hypersetup{allcolors=blue}
\hypersetup{final}

\usepackage[capitalize,nosort,noabbrev]{cleveref}

\Crefname{assumption}{Assumption}{Assumptions}
\Crefname{prop}{Proposition}{Propositions}
\Crefname{thmstepnvi}{Step}{Steps}
\Crefname{thmstepnvii}{Step}{Steps}
\Crefname{lem}{Lemma}{Lemmas}
\Crefname{thm}{Theorem}{Theorems}
\Crefname{fig}{Figure}{Figures}

\crefdefaultlabelformat{#2\textup{#1}#3}
\crefformat{equation}{\textup{(#2#1#3)}}
\crefmultiformat{equation}{(#2#1#3)}%
{ and~(#2#1#3)}{, (#2#1#3)}{, and~(#2#1#3)}
\crefrangeformat{equation}{\textup{(#3#1#4--#5#2#6)}}

\usepackage{verbatim}

\numberwithin{equation}{section}
\numberwithin{figure}{section}
\numberwithin{table}{section}
\theoremstyle{plain}
\newtheorem{thm}{\protect\theoremname}[section]
\theoremstyle{definition}

\newtheorem*{example*}{\protect\examplename}

\theoremstyle{remark}
\newtheorem{rem}[thm]{\protect\remarkname}
\theoremstyle{plain}
\newtheorem{prop}[thm]{\protect\propositionname}
\theoremstyle{plain}

\theoremstyle{plain}
\newtheorem{lem}[thm]{\protect\lemmaname}
\newlist{thmstepnv}{enumerate}{4}
\setlist[thmstepnv]{leftmargin=*,align=left,wide,labelwidth=!,itemindent=!,labelindent=0pt}
\setlist[thmstepnv,1]{label={\itshape {\thmstepname} \arabic*.},ref=\arabic*}
\setlist[thmstepnv,2]{label={\itshape {\thmstepname} {\thethmstepnvi\alph*}.},ref=\thethmstepnvi\alph*}
\setlist[thmstepnv,3]{label={\itshape {\thmstepname\ \alph*.}},ref=\alph*}
\setlist[thmstepnv,4]{label={\itshape {\thmstepname} \arabic*.},ref=\arabic*}
\theoremstyle{plain}
\newtheorem{cor}[thm]{\protect\corollaryname}
\newlist{casenv}{enumerate}{4}
\setlist[casenv]{leftmargin=*,align=left,widest={iiii}}
\setlist[casenv,1]{label={{\itshape\ \casename} \arabic*.},ref=\arabic*}
\setlist[casenv,2]{label={{\itshape\ \casename} \roman*.},ref=\roman*}
\setlist[casenv,3]{label={{\itshape\ \casename\ \alph*.}},ref=\alph*}
\setlist[casenv,4]{label={{\itshape\ \casename} \arabic*.},ref=\arabic*}

\usepackage{anyfontsize}
\usepackage{aligned-overset}
\providecommand{\casename}{Case}
\providecommand{\conjecturename}{Conjecture}
\providecommand{\corollaryname}{Corollary}
\providecommand{\definitionname}{Definition}
\providecommand{\examplename}{Example}
\providecommand{\assumptionname}{Assumption}
\providecommand{\lemmaname}{Lemma}
\providecommand{\propositionname}{Proposition}
\providecommand{\remarkname}{Remark}
\providecommand{\theoremname}{Theorem}
\providecommand{\thmstepname}{Step}

\global\long\def\e{\mathrm{e}}%

\global\long\def\RR{\mathbb{R}}%
 \global\long\def\Pr{\mathbb{P}}%
\global\long\def\NN{\mathbb{N}}%
\global\long\def\ZZ{\mathbb{Z}}%
\global\long\def\clC{\mathcal{C}}%
\global\long\def\clK{\mathcal{K}}%
\global\long\def\sgn{\operatorname{sgn}}%
\global\long\def\Lip{\operatorname{Lip}}%
\global\long\def\dif{\mathrm{d}}%
\global\long\def\eps{\varepsilon}%
\global\long\def\Bin{\operatorname{Bin}}%

\usepackage{csquotes}
\usepackage[backend = bibtex,
style = trad-plain, %
giveninits = false,
isbn = false,
doi = false,
url = false,
eprint = true,
maxbibnames=99,
]{biblatex}
\renewbibmacro{in:}{}
\global\long\def\Rm{\mathbb{R}}

\title{Branching Brownian motion with generation-dependent diffusivity and nonlocal partial differential equations}
\author{Alexander Dunlap\thanks{Department of Mathematics, Duke University, Durham, NC 27708 USA. \url{dunlap@math.duke.edu}.}\and Lenya Ryzhik\thanks{Department of Mathematics, Stanford University, Stanford, CA 94305 USA. \url{ryzhik@math.stanford.edu}.}}

\begin{document}
\maketitle
\begin{abstract}
   We study a voting model on a branching Brownian motion process on $\RR$ in which the diffusivity of each child particle is increased from that of the parent by a factor of $\gamma>1$. The probability distribution of the overall vote is given in terms of the solution to a nonlocal nonlinear PDE\@. We exhibit conditions on the nonlinearity such that the long-time behavior of the distribution undergoes a phase transition in $\gamma$. If~$\gamma$ is sufficiently large, then the long-time distribution converges to uniform. If $\gamma$ is close enough to~$1$, then the long-time distribution depends in a nontrivial way on the location of the initial particle. The limiting dependence is given by a steady-state solution to the nonlocal PDE\@.

Our study gives a probabilistic interpretation of  a class of semilinear nonlocal PDEs.
Interestingly, while the PDE are nonlocal, the underlying random process does not require
any non-local interactions.
\end{abstract}
\section{Introduction}

We consider a branching Brownian motion process on $\RR$, in which offspring particles move faster than their parent by a constant factor $\gamma>1$, but branch at the same rate~$\zeta>0$. We start with a single particle at position $x\in\RR$, which we label with a starting generation number $\ell\in\ZZ$. A particle in generation $k\in\ZZ$ evolves as a Brownian motion with generator~$\frac12\gamma^{2k}\Delta$. Each particle is also equipped with an (independent) exponential clock with rate $\zeta>0$ (not depending on $k$). When the clock rings, the particle  is replaced at its present location by a random positive number of particles (with probability $p_i$ of having $i$ children, $i\ge 1$), each of generation $k+1$, which then undergo the same process, independently of each other. In other words, each particle's offspring move faster than the parent by a factor of $\gamma>1$, but branch at the same rate. Let us emphasize that in this paper we only consider the case $\gamma>1$, when the mobility 
grows from generation to generation. When $\gamma<1$, the particles essentially get nearly frozen after a finite time and the questions of interest are different. 

At a final time $t$, the living particles each cast a ``vote'' according to their current position. A particle at position $y\in\RR$ casts a vote of $\sgn(y)\in\{-1,1\}$ (almost surely). The votes are then propagated up the genealogical tree in accordance with the ``random threshold'' model defined in \cite[\S 3.3]{AHR22}, generalizing the original
majority voting model of~\cite{EFP17}. We refer to~\cite{AHR22} and~\cite{OD19} for a general discussion of the connections between voting models and semilinear
parabolic equations. 
We fix a (deterministic) collection of probabilities~$(\eta_{n,k})$ for~$n\ge1$ and~$1\le k\le n$, with 
\[
\sum_{k=1}^{n} \eta_{n,k} =1
\]
for each $n$.
Each particle on the genealogical tree
votes as follows. Suppose the particle has $n$ children. The particle samples a threshold~$L\in\{1,\ldots,n\}$, independent of the votes and the other thresholds, with 
\[
\Pr(L=k) = \eta_{n,k}.
\]
Then the particle votes ``$1$'' if at least $L$ of its children voted ``$1$,'' and ``$-1$'' otherwise.

\begin{example*}
    The reader will not lose much by taking \begin{equation}\label{eq:majorityetas}\eta_{n,k} \coloneqq \frac12(\delta_{k,\lfloor (n+1)/2\rfloor}+\delta_{k,\lceil (n+1)/2\rceil}),\end{equation} in which case each parent votes along with the majority of its children, breaking ties by a coin flip. This will be a running example throughout the paper. The case where children move at the same speed as the parents ($\gamma=1$ in our notation) and $n=3$ is fixed was studied 
    in \cite{EFP17}.
\end{example*}

Our interest is in the probability $p_\ell(t,x)$ that the initial particle votes $1$, which is a function of the voting time $t$, the initial position $x$, and the starting generation $\ell$. In particular, we study the question of whether $p_\ell(t,x)$ depends on the starting position $x$ in the limit~$t\to\infty$.

Let 
\begin{equation}
u_{\ell}(t,x)=2p_\ell(t,x)-1,
\end{equation}
so that $-1\le u_\ell(t,x)\le 1$. By a renewal analysis following~\cite[(3.11)]{AHR22}, we can write
\begin{equation}\label{eq:uellduhamel}
    u_{\ell}(t,x) = \e^{t(\frac12\gamma^{2\ell}\Delta-\zeta)}\sgn(x)+\zeta\int_0^t\e^{(t-s)(\frac12\gamma^{2\ell}\Delta-\zeta)}(F\circ u_{\ell+1}(s,\cdot))(x)\,\dif s.
\end{equation}
Here, we define
\begin{equation}\label{eq:Fdef}
     F(u) %
     =2\sum_{n\ge 1}p_n \sum_{k=1}^n \eta_{n,k} \Pr_{\frac{1+u}{2}}(X_n \ge k)-1
     =\sum_{n\ge 1}2^{1-n}p_n \sum_{k=1}^n \eta_{n,k}\sum_{j=k}^n\binom n j (1+u)^j(1-u)^{n-j}-1,
\end{equation}
using the notation $\Pr_p$ to denote a probability measure under which $X_n \sim \Bin(n,p)$. 
Note that
\begin{equation}\label{dec1202}
    F(-1)=-1\qquad\text{and}\qquad F(1)=1.
\end{equation}

The equation \cref{eq:uellduhamel} is the Duhamel formula for the solution to the infinite system of PDEs
\begin{align}
    \partial_tu_{\ell}(t,x) &= \frac12\gamma^{2\ell}\Delta u_{\ell}(t,x) + \zeta[F(u_{\ell+1}(t,x))-u_{\ell}(t,x)],&& \ell\in\ZZ,~t>0,~x\in\RR;\label{eq:uellPDE}\\
    u_{\ell}(0,x) &= \sgn(x),&&\ell\in\ZZ,~x\in\RR.\label{eq:uellIC}
\end{align}
indexed by the generation $\ell$.
 
We now note by rescaling the original particle system (and analogously by rescaling the system of PDEs \crefrange{eq:uellPDE}{eq:uellIC}), that
\begin{equation}
    u_{\ell}(t,x)=u_0(t,\gamma^{-\ell}x)\qquad
    \text{for all $\ell\in\ZZ$, $t\ge 0$, and $x\in\RR$.}\label{eq:rescale}
\end{equation}
Writing now $u=u_0$, we see that $u$ satisfies the single \emph{nonlocal} PDE
\begin{align}
    \partial_tu(t,x) &= \frac12\Delta u(t,x) + \zeta[F(u(t,\gamma^{-1}x))-u(t,x)],&& t>0,~x\in\RR;\label{eq:uPDE-wholeline}\\
    u(0,x) &= \sgn(x),&&x\in\RR.\label{eq:uIC-wholeline}
\end{align}
Note that when the mobility growth factor $\gamma=1$, the PDE \cref{eq:uPDE-wholeline} becomes local.
Thus, we will refer to $\gamma=1$ as the ``local case.'' Let us also mention that if were looking for a similar probabilistic description for the
solutions to \crefrange{eq:uPDE-wholeline}{eq:uIC-wholeline}, with an initial condition in \cref{eq:uIC-wholeline}
different from $\sgn(x)$, then the initial conditions for $u_\ell$ in \crefrange{eq:uellPDE}{eq:uellIC} would be 
$\ell$-dependent. It is straightforward to adapt the voting model to that situation. We use the $\sgn(x)$-based voting rule for the sake of simplicity.

Our analysis will apply to the case when $F$ is an odd function, which is the case whenever 
\[
\eta_{n,k} = \eta_{n,n-k+1}\qquad\text{for each $n,k$}.
\] 
This is satisfied for example by the $\eta_{n,k}$ defined in \cref{eq:majorityetas}.
In this case, $u(t,\cdot)$ remains an odd function for all $t\ge 0$, and so we can recast the problem \crefrange{eq:uPDE-wholeline}{eq:uIC-wholeline} 
as a boundary value problem on the half line:
\begin{align}
    \partial_tu(t,x) &= \frac12\Delta u(t,x) + \zeta[F(u(t,\gamma^{-1}x))-u(t,x)],&& t,~x>0;\label{eq:uPDE}\\
    u(t,0) &= 0,&&t>0;\label{eq:uDirichlet}\\
    u(0,x) &= 1,&&x>0;\label{eq:uIC}
\end{align}

For an odd function $F$, we define
\begin{equation}\label{dec1204}
\Xi_F\coloneqq \inf\{v>0\mid F(v)\le v\}.
\end{equation}
We note that $\Xi_F$ is a fixed point of $F$ and hence a constant solution to \cref{eq:uPDE}.
By \cref{dec1202}, we see that $\Xi_F\le 1$, since $1$ and $-1$ are fixed points of $F$. %
If $F'(0)\le 1$, it may even be the case that $\Xi_F = 0$.

\begin{rem}\label{rem-dec1202}
Let us recall that in the local case $\gamma=1$, a non-trivial steady solution to \cref{eq:uPDE} exists as soon as $F'(0)>1$, so that $F(u)>u$
in a neigborhood of $u=0$ and $\Xi_F>0$. In particular, if $\Xi_F=1$, such solution will tend to $1$ as $x\to+\infty$, so that particles starting
sufficiently far away from $x=0$ will have a very high probability to vote $1$. Our results below show that this phenomenology persists if 
the particle mobility growth rate $\gamma$ is not too large, but breaks if $\gamma$ is sufficiently large. 
\end{rem}

\begin{rem}
The above construction gives a probabilistic interpretation to solutions to the
nonlocal PDE \cref{eq:uPDE-wholeline}.  {The nonlocality} can be thought of as a multiplicative
convolution  {(with a shifted delta function in the present case)}. A perhaps surprising aspect is that such nonlocal equations can be
realized directly by branching Brownian motion with no jumps or non-local interactions
but simply introducing a generation-dependent particle mobility. We should mention 
that~\cite{ABBP19,ABP17} 
discovered a probabilistic interpretation for non-local PDE of the form 
\begin{equation}\label{dec2202}
\partial_tu=\Delta u+u(1-\phi* u),
\end{equation}
with an additive rather than multiplicative  
convolution. However, there the PDE arose in a hydrodynamic limit, and the underlying probabilistic 
model did involve nonlocal interactions, unlike
the construction in the present paper. On the other hand, it would be interesting to find
a connection between the
Brownian motion for particles with time-dependent masses considered in those references
and the branching process with generation-dependent mobility that we consider here. 
It is not clear if a version of \cref{dec2202} admits a probabilistic interpretation 
via a branching process without non-local interactions. 
\end{rem}

Our first result is the following.
\begin{thm}\label{thm:convergence}
    Assume that $F$ is an odd function. The limit
    \begin{equation}U_{\infty}(x)\coloneqq\lim_{t\to\infty} u(t,x)\label{eq:limitexists}\end{equation}
    exists. Moreover, $U=U_{\infty}|_{(0,\infty)}$ is the maximal solution to the nonlocal elliptic PDE problem
\begin{align}
    &0 = \frac12\Delta U(x) + \zeta[F(U(\gamma^{-1}x))-U(x)],&& x>0;\label{eq:nonlocalelliptic}\\
    &U(0) = 0;\label{eq:nonlocalellipticboundary}\\
    &U(x) \le 1,&&x>0.\label{eq:nonlocalellipticlt1}
\end{align}
\end{thm}

One solution to \crefrange{eq:nonlocalelliptic}{eq:nonlocalellipticlt1} is of course given by $U\equiv 0$. Our main question concerns whether this is the only solution, or if there is a greater solution (the greatest of which would be $U_\infty$). The answer depends on $F$ and $\gamma$, and is the subject of our second result.
\begin{thm}\label{thm:dichotomy}
    Assume that $F$ is an odd function.
    \begin{enumerate}
        \item \label{enu:super} If $\gamma>\Upsilon_F\coloneqq \sup\limits_{v\in (0,1]}F(v)/v$, then $U_\infty\equiv 0$.
        \item \label{enu:sub} If $\gamma<F'(0)$, then \begin{equation}\liminf\limits_{x\to+\infty}U_\infty(x)\ge \Xi_F.\label{eq:Uinftynottozero}\end{equation}
    \end{enumerate}
\end{thm}
\begin{rem}  %
Let us recall that when $F'(0)>1$, the state~$u=0$ is an unstable
equilibrium of the ODE 
\begin{equation}
\dot u=F(u)-u,
\end{equation}
This leads to the ``finite front width" picture in the local case~$\gamma=1$ described in \cref{rem-dec1202}. On the other hand,
Theorems~\ref{thm:convergence} and~\ref{thm:dichotomy}(\ref{enu:super}) show that if  $F'(0)>1$ but $\gamma>\Upsilon_F$ (that is,  the 
particle mobility grows sufficiently fast), then we see growth in time of the region of the starting particle for which the ultimate vote is $1$ and $-1$ with nearly equal probabilities.
In other words, a sufficiently high mobility growth rate  stabilizes the state $u=0$,
a situation that is impossible with $\gamma=1$.
\end{rem}
 
\begin{rem} A natural question is whether the thresholds in the two parts of \cref{thm:dichotomy} match.
    If the Fisher--KPP-type condition
    \begin{equation}F(v)\le F'(0)v \qquad\text{for all }v\in [0,1],\label{eq:KPP}
    \end{equation} holds,
    then $\Upsilon_F= F'(0)$, and \cref{thm:dichotomy} leaves only a single critical value $\gamma = F'(0)$ at which we do not determine the behavior. 
To avoid some confusion of terminology, let us recall that the nonlinearity that appears in \cref{eq:uPDE}  in the local case $\gamma=1$ is~$F(u)-u$
and not~$F(u)$. Moreover, in the  \ordinalstringnum{\getrefnumber{enu:sub}} case of \cref{thm:dichotomy}, when the solutions to~\crefrange{eq:nonlocalelliptic}{eq:nonlocalellipticlt1}  are nontrivial, the solutions do not tend to $0$ as $x\to+\infty$.
Thus, the assumption~\cref{eq:KPP}
on $F(u)$ does not at all mean that the evolution in \cref{eq:uPDE} is expected to be of the Fisher--KPP type. Rather, the oddness assumption
on $F(u)$ means that we are in a bistable (or multistable) situation in the terminology of the reaction-diffusion equations.
\end{rem}
\begin{prop}\label{prop:majorityiskpp}
    The majority voting procedure given by \cref{eq:majorityetas} yields an $F$ satisfying \cref{eq:KPP}, for any choice of the $p_n$s. Moreover, in this case we have
    \begin{equation}\label{eq:Fprime0}
        F'(0) = \sum_{n\ge 1}p_n\left(2^{1-n} \lceil n/2\rceil \binom n {\lfloor n/2\rfloor}\right),
    \end{equation}
    and, if there is some $n\ge 3$ with $p_n>0$, then $F'(0)>1$ and
        $\Xi_F = 1$.%
\end{prop}
\begin{rem}
We note that by Sterling's formula, the term in parentheses in \cref{eq:Fprime0} is asymptotic to $\sqrt{2n/\pi}$ for large $n$.
\end{rem}

Theorems~\ref{thm:convergence} and~\ref{thm:dichotomy}(\ref{enu:super}) show that if $\gamma>\Upsilon_F$, then $u(t,x)\to 0$ as $t\to+\infty$ 
for each~$x\in\Rm$ fixed.
In that case, one may ask how rapidly the region where $u(t,x)\approx 0$ spreads. %
To this end, for $f,\gamma>1$, define
\begin{equation}
    \Sigma(f,\gamma)\coloneqq \sup_{\omega\in (0,1]} \frac{1-f \gamma^{-\omega}}{\omega}.
    \label{eq:Sigmadef}
\end{equation}
The next theorem says that, for speeds $\underline{c}_*$ and $\overline{c}_*$ defined in \cref{eq:cdefs} below, $U(t,x)\ll 1$ in the regime $x\ll \exp(\underline c_*t)$ %
while~$U(t,x)$ is bounded below in the regime
$x\gtrsim\exp(ct)$ for any~$c>\overline c_*$.
\begin{thm}\label{thm:speed}
    Let $F$ be an odd function, $F'(0)>1$, and $\zeta >0$, and define \begin{equation}\underline c_*\coloneqq \zeta\Sigma(\Upsilon_F,\gamma)\qquad\text{and}\qquad\overline c_*\coloneqq \zeta\Sigma(F'(0),\gamma).\label{eq:cdefs}\end{equation}
    \begin{enumerate}
        \item \label{enu:speedsuper}If $\Sigma(\Upsilon_F,\gamma)> 0$, then there is a $\overline{v}\in\mathcal{C}([0,\infty))$ such that $\overline v(0)=0$ and 
        \begin{equation}u(t,x)\le \overline v(x\e^{-c_*t})\qquad\text{for all $t\ge1$ and $x\ge 0$}.\label{eq:speed-super}\end{equation}
        \item \label{enu:speedsub} If $\nu\ge 0$ and $\nu > \zeta\Sigma(F'(0),\gamma)$, then there is a $\underline{v}\in \clC([0,\infty))$ such that $\lim\limits_{x\to\infty} \underline v(x)=\Xi_F$ and
            \begin{equation}
                u(t,x)\ge\underline v(x\e^{-\nu t})\qquad\text{for all $t,x\ge0$.}\label{eq:speed-sub}
            \end{equation}
    \end{enumerate}
\end{thm}
\begin{rem}
    If \cref{eq:KPP} holds so that $\Sigma(\Upsilon_F,\gamma)=\Sigma(F'(0),\gamma)>0$, then the rates in \cref{eq:speed-super,eq:speed-sub} match,
    so $\underline c_*=\overline c_*\eqqcolon c_*$. In this case, we may conjecture that $u(t,e^{c_*t}\cdot)$ 
    converges to a nontrivial limiting function as $t\to +\infty$. We should also mention that  the definition 
    of~$\Sigma(F'(0),\gamma)$ is reminiscent of the Freidlin--G\"artner formula~\cite{FG79} for the speed of propagation of the solutions
    to the Fisher-KPP type equations in the local case.
\end{rem}
\begin{rem}
    If $\gamma < F'(0)$, then $\Sigma(\Upsilon_F,\gamma)<\Sigma(F'(0),\gamma) <0$ and \cref{eq:speed-sub} (with $\nu = 0$) is simply a restatement of \cref{thm:dichotomy}(\ref{enu:sub}).
\end{rem}

\begin{rem}The exponential in time rates of spreading have been previously observed  in a PDE with a fractional
Laplacian and a Fisher--KPP nonlinearity in~\cite{CCR-per,CR-CRAS,Ca-Roq13,Cou-Roq12}, as well as with 
a non-local diffusion~\cite{HH20}. Interestingly, the situation with fractional diffusions
and bistable type of nonlinearities (i.e.\ those we consider here) is different: the front spreads linearly in time and its width does not grow~\cite{MRS14}.
 Another instance when spreading is super-linear in time is in reaction-diffusion
models of mobility-structured populations~\cite{BMR15,BHR17,CHMTD22,HPS18}.  However, in none of these models is the unstable
state stabilized, regardless of the mobility of the individuals, in contrast to the model considered in this paper.  
\end{rem}

\begin{rem}
    In this note we consider the random threshold model, rather than the ``random outcome model'' (see~\cite{AHR22,OD19}). With the random threshold model, 
    the nonlinearity~$F$ is guaranteed to be monotone increasing, while this is not guaranteed for the random outcome model. In the local
    case $\gamma=1$ studied in previous works, this distinction is less important, because the nonlinearity in the equation is really $v\mapsto \zeta [F(v)-v]$, so if one wants to create a nonlinearity $\zeta[F(v)-v]= G$, one can simply choose $\zeta$ large enough that $\frac1\zeta G(v)+v$ is increasing. However, in the present setting of $\gamma>1$, the nonlinear and linear terms in \cref{eq:uellPDE}/\cref{eq:uPDE-wholeline}/\cref{eq:uPDE} are applied to $u$ evaluated at different locations. In this case, the assumption that $F$ is increasing is important to obtain a comparison principle for the equation (\cref{prop:comparisonprinciple} below). This difficulty has been also observed for branching random walks in~\cite{KRZ23}.
\end{rem}

\begin{rem}
We should mention that the branching model in this paper is somewhat related to the interesting 
two-type models considered in~\cite{Bell-Mall21,Holzer14,Holzer16,Holzer-Scheel14},
where the off-spring may have a different diffusivity from the parent particle. However, in those models only two mobilities are possible and the
equations are local in space, leading to linear-in-time spreading and not changing the stability of the unstable equilibrium. 
\end{rem}

\subsection{Propagation of particle positions}\label{subsec:ppp}

The following probabilistic model, connected to that introduced in~\cite{KRZ23},
is equivalent to the one described above. Consider the same branching Brownian motion process (with starting position $x\in\RR$ and acceleration factor $\gamma>1$). We fix the probabilities $(\eta_{n,k})$ as above. But instead of each particle casting a vote, in this model we associate to each particle a value in $\RR$. The particles that are still living at time $t$ receive a value equal to their position on the line. A parent particle with $n$ children then receives the $k$th largest of its children's values with probability $\eta_{n,k}$. Let~$X_t$ denote the value associated to the original parent particle in this way. A moment's thought indicates that \begin{equation}u(t,x) = 2\Pr^x(X_t > 0)-1,\label{eq:uasprobofX}\end{equation} where $\Pr^x$ denotes the probability measure in which the initial particle starts at $x$. Indeed, the voting scheme described above is the same as the process just described, except that the $\sgn$ function is applied to the living particles' values before the values are propagated up the genealogical tree. Since the propagation scheme only depends on the ordering of the values, it is equivariant under the application of a monotone function.

Now we suppose that fix the initial condition $x=0$. In this case, we note that by translation invariance, the probability that $X_t>y$ is the same as the probability that $X_t >0$ when the starting position was chosen as $x=-y$. By \cref{eq:uasprobofX}, this means that 
\begin{equation}\label{eq:uastranslatedprob}
    u(t,x) = 2\Pr^0(X_t>-x)-1.
\end{equation}
Thus we have the following corollary of \cref{thm:convergence,thm:dichotomy}(\ref{enu:sub}).
\begin{cor}\label{cor:convindistribution}
    Suppose that $F$ is an odd function and $\Xi_F = 1$. Then (under the probability measure $\Pr^0$) the family of random variables $(X_t)_{t\ge 0}$ converges in distribution as $t\to\infty$ to a random variable $X$ with cdf
    \begin{equation}
        \Pr(X\le x) = \frac12(1-U_\infty(-x)).
    \end{equation}
\end{cor}
\begin{rem}
    One can naturally ask about analogues of \cref{cor:convindistribution} for similar nonlinear functionals of Gaussian fields. This is somewhat in analogy to the behavior of the maximum of branching Brownian motion, results on which have been extended to the maximum of the two-dimensional Gaussian free field and other log-correlated fields. The latter subject is extremely rich, and we do not attempt to survey it here. We refer to~\cite{Bis20} and its many references for an introduction.
\end{rem}

\subsection*{Plan of the paper}
In \cref{sec:localPDE}, we introduce and solve a problem that is similar to that described above but that can be modeled with a local PDE\@. This illustrates much of the phenomenology and ideas in the proofs of our main theorems, but with less technical complexity. In \cref{sec:nonlocal-Cauchy}, we prove \cref{thm:convergence,thm:dichotomy,thm:speed}. In \cref{sec:majvote}, we prove \cref{prop:majorityiskpp}.

\subsection*{Acknowledgments}
AD was partially supported by the NSF Mathematical Sciences Postdoctoral Research Fellowship Program under grant DMS-2002118 at NYU Courant.
LR was supported by NSF grants DMS-1910023 and DMS-2205497 and by~ONR grant N00014-22-1-2174.

\section{The local-PDE version of the problem}\label{sec:localPDE}

In this section, which is logically independent from the rest of the paper, we show the key ideas in the proofs of \cref{thm:convergence,thm:dichotomy} by considering a   
simpler model. It also involves a branching Brownian motion, but rather than particles speeding up every time they branch, they instead diffuse with time-dependent (but generation-independent) generator $\frac12\gamma^{2(\tau+t)}\Delta$. Here, $\tau$ is a fixed constant, analogous to the starting generation label~$\ell$ in the generation-dependent problem. As before, let $p(t,x)$ be the probability that the original particle votes ``$1$,'' and let $u_\tau(t,x)=2p(t,x)-1$. A scaling argument,
analogous to that leading to \cref{eq:rescale}, implies that
\begin{equation}\label{eq:rescale-cts}
u_\tau(t,x) = u_0 (t,\gamma^{-\tau} x).
\end{equation}
Moreover, a renewal analysis similar to that described in the introduction shows that
\begin{equation}\label{dec1302}
\begin{aligned}
    u_0(t,x) &= \exp \Big\{\int_0^t\Big(\frac12\gamma^{2s}\Delta-\zeta\Big)\,\dif s\Big\}\sgn(x)\\
             &\qquad + \zeta\int_0^t \exp\Big\{\int_0^{t-s}\Big(\frac12\gamma^{2r}\Delta-\zeta\Big)\,\dif r\Big\}(F\circ u_{t-s}(s,\cdot))(x)\,\dif s\\
             &=\exp \left\{\frac{\gamma^{2t}-1}{4\log\gamma}\Delta-t\zeta\right\}\sgn(x)\\
             &\qquad + \zeta\int_0^t \exp\left\{\frac{\gamma^{2(t-s)}-1}{4\log\gamma}\Delta-(t-s)\zeta\right\}(F\circ u_0(s,\gamma^{-(t-s)}\cdot))(x)\,\dif s.
\end{aligned}
\end{equation}
Now if we define
\begin{equation}\label{eq:vfromu}v(t,y)=u_0(t,\gamma^{t}y),\end{equation}
then \cref{dec1302} becomes
\begin{align*}
    v(t,y)  
             &=\exp \left\{\frac{\gamma^{2t}-1}{4\log\gamma}\Delta-t\zeta\right\}\sgn(\gamma^{t}y)\\
             &\qquad + \zeta\int_0^t \exp\left\{\frac{\gamma^{2(t-s)}-1}{4\log\gamma}\Delta-(t-s)\zeta\right\}(F\circ v(s,\gamma^{-t}\cdot))(\gamma^t x)\,\dif s\\
             &=\exp \left\{\frac{1-\gamma^{-2t}}{4\log\gamma}\Delta-t\zeta\right\}\sgn(y)\\
             &\qquad + \zeta\int_0^t \exp\left\{\frac{\gamma^{-2s}-\gamma^{-2t}}{4\log\gamma}\Delta-(t-s)\zeta\right\}(F\circ v(s,\cdot))(x)\,\dif s.
\end{align*}
Differentiating, we see that $v$ solves the Cauchy problem
\begin{align}
    \partial_t v(t,y) &= \frac12 \gamma^{-2t} \Delta v(t,y) + \zeta f(v(s,y)),&& t>0,~y\in\RR;\\
    v(0,y) &= \sgn(y),&&y\in\RR,
\end{align}
where we have defined
\[f(v) = F(v)-v.\]
Now we define $b = \log\gamma>0$ and, for $\nu\in\RR$,
\[u^{(\nu)}(t,x) = v(t,\e^{-(b-\nu)t}x),\]
(so in particular $u^{(0)} = u_0\eqqcolon u$), then $u^{(\nu)}$ satisfies  
\begin{align}
    \partial_t u^{(\nu)}(t,x) &= \frac12 \e^{-2\nu t}\Delta u^{(\nu )}(t,x) - (b-\nu ) x \partial_x u^{(\nu )}(t,x) + \zeta f(u^{(\nu )}(t,x)),&& t>0,x\in\RR;\label{eq:ulocalPDE-wholeline}\\
    u^{(\nu )}(0,x) &= \sgn(x),&&x\in\RR,\label{eq:ulocalIC-wholeline}
\end{align}
Assuming, as in the introduction, that $F$ and hence $f$ is an odd function, we see that the solution $u$ will remain an odd function as well, and so we can recast \crefrange{eq:ulocalPDE-wholeline}{eq:ulocalIC-wholeline} as the half-line PDE problem
\begin{align}
    \partial_t u^{(\nu )}(t,x) &= \frac12\e^{-2\nu t} \Delta u^{(\nu )}(t,x) - (b-\nu ) x \partial_x u^{(\nu )}(t,x) + \zeta f(u^{(\nu )}(t,x)),&& t,x>0;\label{eq:ulocalPDE}\\
    u^{(\nu )}(0,x) &= 1,&&x>0;\label{eq:ulocalIC}\\
    u^{(\nu )}(t,0) &= 0,&&t>0.\label{eq:ulocalDirichlet}
\end{align}

The problem \crefrange{eq:ulocalPDE}{eq:ulocalDirichlet} is a (local, unlike \crefrange{eq:uPDE}{eq:uDirichlet}) second-order PDE, and satisfies the comparison principle. When $\nu =0$, the PDE \cref{eq:ulocalPDE} is autonomous. Therefore, applying the comparison principle to $u(\cdot+h,\cdot)-u$, which is initially negative, we see 
that~$t\mapsto u(t,x)$ is decreasing in time, so $U_\infty(x)\coloneqq \lim\limits_{t\to\infty} u(t,x)$ exists. Since $0$ is a (sub)solution to \cref{eq:ulocalPDE}, we know that $U_\infty(x)\ge 0$ for all $x$. Thus, our main question in studying \crefrange{eq:ulocalPDE}{eq:ulocalDirichlet} is whether~$U_\infty\equiv 0$ or if $U_\infty$ is positive for some positive $x$.
We will prove somewhat more general results on $u^{(\nu )}$ for general $\nu $, which, in addition, 
give information about the rates of spreading of $u^{(0)}$. Interesting, the rates in \cref{dec1306,dec1310} below do not take the same form as the rates in \cref{eq:cdefs} for the nonlocal problem.

First we prove the following analogue of \cref{thm:dichotomy}(\ref{enu:super}) and \cref{thm:speed}(\ref{enu:speedsuper}):
\begin{prop}\label{prop:dec1302}
    Define %
    \begin{equation}\label{eq:Rfdef}
        \Upsilon_f \coloneqq \sup_{v\in [0,1]} \frac{f(v)}{v}.
    \end{equation}
    If $\zeta \Upsilon_f <b-\nu $, then
    \begin{equation}\label{dec1304}
        \lim_{t\to\infty} u^{(\nu )}(t,x)=0\qquad\text{for each fixed }x\ge 0.
    \end{equation}
\end{prop}
\cref{prop:dec1302} implies that for any 
\begin{equation}\label{dec1306}
\nu<\underline c:=b-\zeta\Upsilon_f,
\end{equation}
and $x>0$ fixed we have 
\[
u_0(t,e^{\nu t}x)=v(t,e^{-(b-\nu) t}x)=u^{(\nu)}(t,x)\to 0,~~\hbox{as $t\to+\infty$.}
\]
Thus, the state $u=0$ invades the positive half-line, roughly speaking, at least at the exponential rate $\exp(\underline c t)$. 
\begin{proof}
 Let $\delta\in(0,b-\nu-\zeta\Upsilon_f)$. %
 For fixed $\xi>0$, let $v_{\delta,\xi}(t,x) = \e^{-\delta t}\xi x$. We note that
 \begin{equation}
\begin{aligned}
  \frac12 \e^{-2\nu t}\Delta v_{\delta,\xi}(t,x)& - (b-\nu )x\partial_xv_{\delta,\xi}(x) + \zeta f(v_{\delta,\xi}(t,x)) \le (\zeta \Upsilon_f -b+\nu ) v_{\delta,\xi}(t,x)\\
&\le    -\delta v_{\delta,\xi}(t,x)=\partial_t v_{\delta,\xi}(t,x) ,
\end{aligned}
\end{equation}
by our choice of $\delta$. We conclude that $v_{\delta,\xi}$ is a supersolution to \crefrange{eq:ulocalPDE}{eq:ulocalDirichlet}. Since $u$ is smooth at positive times (being a solution to the whole-line problem \crefrange{eq:ulocalPDE-wholeline}{eq:ulocalIC-wholeline}), for any $t>0$, there is some $\xi$ such that $u^{(\nu )}(t,x)\le v_{\delta,\xi}(t,x)$ for all $x\ge 0$. The comparison principle then implies \cref{dec1304}. %
\end{proof}

Now we prove the following analogue of \cref{thm:dichotomy}(\ref{enu:sub}) and \cref{thm:speed}(\ref{enu:speedsub}):
\begin{prop}\label{prop:dec1304}
    If $\nu \ge 0$ and $b-\nu <\zeta f'(0)$, then there is a function $v=v(x)$ on $(0,\infty)$ such that 
    \begin{equation}\label{dec1316}
    \lim\limits_{x\to +\infty} v(x)= \Xi_F,
    \end{equation}
   and $u^{(\nu )}(t,x)\ge v(x)$ for each $t,x\ge 0$.
\end{prop}
\cref{prop:dec1304} implies that for any 
\begin{equation}\label{dec1310}
\nu>\overline c:=b-\zeta f'(0),
\end{equation}
and $x>0$ fixed we have 
\[
u_0(t,e^{\nu t}x)=v(t,e^{-(b-\nu) t}x)=u^{(\nu)}(t,x)\ge v(x). 
\]
Thus, the state $u=0$ invades the positive half-line, roughly speaking, not faster than at the exponential rate $\exp(\overline c t)$. Note that 
$\underline c=\overline c$ in the Fisher-KPP case when $\Upsilon_f=f'(0)$.

\begin{proof}
    Fix $B>0$, to be chosen later, and let $W$ be a solution to the ODE 
    \begin{equation}
        W'(Y) = \frac1B f(W(Y)),\qquad Y\in\RR.
    \end{equation}
    with $W(0) \in (0,\Xi_F)$.
    This means in particular that $\lim\limits_{Y\to-\infty} W(Y) = 0$ and $\lim\limits_{Y\to+\infty} W(Y) = \Xi_F$.
    Now if we define
    \[
        w(y) = W(\log y),
    \]
    then  
    \begin{equation}
        w'(y) = \frac1y W'(\log y) = \frac1{By}f(w(y)),\label{eq:wderiv}
    \end{equation}
 and we also have
 \begin{equation}\label{dec1312}
 \lim_{y\searrow 0}w(y)=0\qquad\text{and}\qquad\lim_{y\to+\infty}w(y)=\Xi_F.
\end{equation}
    We would like to use $w$ to build a subsolution to \cref{eq:ulocalPDE}. The identity \cref{eq:wderiv} will be sufficient to understand the second and third terms on the right side. To deal with the Laplacian, we need to understand the second derivative of $w$. So we differentiate \cref{eq:wderiv}, and then use~\cref{eq:wderiv} again, to obtain
    \begin{equation}
        w''(y) = \frac{f'(w(y))w'(y)}{By}-\frac1{By^2}f(w(y))
        = (f'(w(y))-B)\frac{f(w(y))}{B^2y^2}.\label{eq:wsecondderiv}
    \end{equation}
    We define the subsolution  to be $v(x) = w(\eta x)$, for some $\eta>0$ to be determined. By~\cref{eq:wderiv,eq:wsecondderiv}, we have, for any $t\ge 0$,
    \begin{multline}
        \frac12e^{-2\nu t} v''(x) - (b-\nu )x v'(x) + \zeta f(v(x)) = \frac12 \eta^2\e^{-2\nu t} w''(z) - (b-\nu )z w'(z) + \zeta f(w(z))\\
                                                              = \Big(\eta^2\e^{-2\nu t}\cdot\frac{f'(w(z))-B}{2B^2z^2} - \frac {b-\nu }B +  \zeta\Big) f(w(z)),\label{eq:computeforv}
\end{multline}
    where we have defined $z=\eta x$.
    Now we choose $B\in ((b-\nu )/\zeta,f'(0))$, which means that 
    \[
    -\frac{b-\nu }B+\zeta >0\qquad\text{ and }\qquad\inf\limits_{z>0}\frac{f'(w(z))-B}{2B^2z^2}>-\infty.
    \]
    The second inequality above holds since $\lim\limits_{z\searrow 0}w(z)=0$ by \cref{dec1312} and $f'(0)-B>0$. %
    Thus, we can choose $\eta$ sufficiently small that (using the assumption that $\nu \ge 0$)
    \[
    \Big(\eta^2\e^{-2\nu t}\cdot\frac{f'(w(z))-B}{2B^2z^2} - \frac {b+h}B +  \zeta\Big) f(w(z))>0\qquad\text{for all }t,z>0,
    \]
    which means that
    \[\frac12\e^{-2\nu t} v''(x) - (b-\nu )x v'(x) + \zeta f(v(x)) \ge 0\qquad\text{for all }t,x> 0\] by \cref{eq:computeforv}.
    By the comparison principle, we therefore have that $u^{(\nu )}(t,x)\ge v(x)$ for all~$t,x\ge 0$. Moreover, the limiting behavior
    as $x\to+\infty$ in \cref{dec1316} holds for $v(x)$ because of~(\ref{dec1312}). 
\end{proof}

\section{The nonlocal Cauchy problem}\label{sec:nonlocal-Cauchy}

In this section we prove our main results. We first discuss the comparison principle for the non-local problem in \cref{sec:compar}.
It is used in \cref{sec:monot} to prove \cref{thm:convergence} on the convergence of the solution to a steady state. Part~\ref{enu:super} of \cref{thm:dichotomy}
on the triviality of the steady states when the mobility growth rate is sufficiently large is proved in \cref{sec:super}. We also prove
there \cref{thm:speed}(\ref{enu:speedsuper}), the lower bound on the exponential spreading rate of the state $u\approx 0$. The \ordinalstringnum{\getrefnumber{enu:sub}} parts of \cref{thm:dichotomy,thm:speed} are proved in \cref{sec:sub}. \cref{sec:analyt}
contains the analytic constructions used in the aforementioned proofs.

We will study $u(t,x)$ through the nonlocal Cauchy problem \crefrange{eq:uPDE}{eq:uIC} on the half-line. In order to clarify our analysis, 
we summarize in the following proposition the facts about $F$ that we will use.
\begin{prop}
    The function $F(v)$ is increasing, $\clC^2$ on $[0,1]$, and satisfies $F(v)>v$ for all~$v\in (0,\Xi_F)$.
\end{prop}

\subsection{The cut-off problem and the comparison principle}\label{sec:compar}

It will be technically useful to construct a cut-off problem. Let $L>0$. We can consider a modification of the dynamics described in the introduction, if a particle of generation $k$ reaches a point in $\{\pm \gamma^k L\}$, it is immediately stopped, henceforth no longer moving nor reproducing. Thus, it will end up voting ``$1$'' if it is stopped at $+\gamma^k L$ and ``$-1$'' if it is stopped at~$-\gamma^k L$. A straightforward adaptation of the derivation indicated in the introduction shows that, if $p$ is the probability that the original ancestor, in generation $\ell$, votes~$1$ when the process is started with one particle at $\gamma^\ell x$, then $u^{[L]}(t,x) = 2p-1$ does not depend on $\ell$ and satisfies the problem
\begin{align}
    \partial_tu^{[L]}(t,y)&=\frac12 \Delta u^{[L]}(t,y)+\zeta[ F(u^{[L]}(t,\gamma^{-1}y))- u^{[L]}(t,y)],&&\text{$t\ge 0$, $x\in (0,L)$;}\label{eq:uLPDE}
    \\
    u^{[L]}(t,0) &= 0,&& t> 0;\label{eq:uLleft}
    \\
    u^{[L]}(t,L) &= 1,&& t > 0;\label{eq:uLright}
    \\
    u^{[L]}(0,x) &= 1,&& x\in [0,L].\label{eq:uLic}
\end{align}

The probability that there is a $k\ge 0$ such that a particle of generation $k$ reaches $\pm \gamma^k L$ before a fixed time $t$ goes to $0$ as $L\to \infty$. This immediately implies that
\begin{equation}\label{eq:vLtov}
    \lim_{L\to\infty}u^{[L]}(t,x)=u(t,x)\qquad\text{for each }t,x\ge 0.
\end{equation}

The fact that $F$ is increasing means that the cut-off problem \crefrange{eq:uLPDE}{eq:uLic} admits a comparison principle. In fact, we will state the comparison principle more generally, as we will have occasion to use it for nonlinearities that are not of the form \cref{eq:Fdef}. %
\begin{prop}\label{prop:comparisonprinciple}
    Let $G$ be an increasing Lipschitz function.
    Let $\overline u$ and $\underline u$, assumed to be $\clC^1$ in time and piecewise $\clC^2$ in space, be super- and sub-solutions, respectively, to \cref{eq:uLPDE} with $F$ replaced by $G$. By this we mean that, for each $t,x$ at which the functions are $\clC^2$, we have
    \begin{equation}
        \partial_t\overline u(t,x)\ge\frac12 \Delta \overline u(t,x)+\zeta[ G(\overline u(t,\gamma^{-1}x))- \overline u(t,x)],\qquad\text{$t\ge 0$, $x\in (0,L)$,}\label{eq:supersoln}
    \end{equation}
and
\begin{equation}
        \partial_t\underline u(t,x)\le\frac12 \Delta \underline u(t,x)+\zeta[ G(\underline u(t,\gamma^{-1}x))- \underline u(t,x)],\qquad\text{$t\ge 0$, $x\in (0,L)$,}\label{eq:subsoln}
    \end{equation}
    and moreover that, at any $t,x$ for which $\overline u$ (resp. $\underline u$) is not $\clC^2$ in space, we have \begin{equation}\partial_x \overline u(t,x-)>\partial_x\overline u(t,x+) \qquad\qquad\text{(resp. }\partial_x \underline u(t,x-)<\partial_x\underline u(t,x+)\text).\label{eq:derivcond}\end{equation}
Suppose also that $\overline u(t,y)\ge \underline u(t,y)$ whenever $y=0$, $y=L$, or $t=0$. Then we have 
\begin{equation}\label{dec1318}
\hbox{$\overline u(t,y)\ge \underline u(t,y)$ for all $t\ge 0$ and $x\in [0,L]$.}
\end{equation}
\end{prop}
\begin{proof}
    The proof proceeds in a standard way for a comparison principle for a parabolic equation, but since our equation is nonlocal, we give the details.

    Let \begin{equation}\mu >\zeta[\Lip(G)-1]\label{eq:mudef}\end{equation} and define
    \[
    w(t,x) = \e^{-\mu t}(\overline u-\underline u)(t,x).\]
    Now fix $\eps>0$ and let \[T = \inf\left\{t\ge 0\ :\ \inf_{x\in [0,L]}w(t,x)\le -\eps\right\}.\] Suppose for the sake of contradiction that $T<\infty$.
    The boundary conditions on $\overline u$ and $\underline u$ imply that $T>0$, that 
    $w(T,\cdot)$ achieves its minimum at some $X\in (0,L)$, and that \begin{equation}\label{eq:wisminminuseps-general}w(T,X)=\inf\limits_{x\in [0,L]}w(T,x)=-\eps.\end{equation} Moreover, \cref{eq:derivcond} implies that $w$ is $\clC^1$ in time and $\clC^2$ in space at $T,X$. %
    By the definitions of $w$, $T$ and $X$, along with \cref{eq:supersoln,eq:subsoln,eq:wisminminuseps-general}, we have
    \begin{equation}
    \begin{aligned}0&\ge \partial_t w(T,X) \\&\ge -\mu w(T,X) + \frac12\Delta w(T,X) + \zeta[\e^{-\mu t}(G(\overline u(T,\gamma^{-1} X))-G(\underline u(T,\gamma^{-1} X)))-w(T,X)].%
    \end{aligned}
    \label{eq:0gt}
    \end{equation}
    We have $\Delta w(T,X)\ge 0$ since the minimum is achieved at $X$, and $w(T,X)=-\eps$.
        Also, using the assumption that $G$ is monotone and Lipschitz, we can write \begin{align*}\e^{-\mu t}(G(\overline u (T,\gamma^{-1}X))-G(\underline u(T,\gamma^{-1}X)))&\ge -\e^{-\mu t}\Lip(G)(\overline u(T,\gamma^{-1}X)-\underline u(T,\gamma^{-1}X))_-\\&\ge -\Lip(G)\eps ,\end{align*}
        where we use the notation $x_- = -(x\wedge 0)$, and the second inequality is by \cref{eq:wisminminuseps-general}. Using these observations
        in \cref{eq:0gt}, we obtain
        \[0\ge (\mu -\zeta\Lip(G)+\zeta)\eps,\]
        contradicting \cref{eq:mudef}.
\end{proof}

\subsection{Monotonicity and convergence}\label{sec:monot}

In this subsection we prove \cref{thm:convergence}. We do this by showing that $u(t,x)$ is decreasing in $t$, hence must approach a limit. The comparison principle will imply that the solution cannot cross any subsolution, but the limit itself will be a solution, so the limit will in fact be the maximal solution.

\begin{prop}\label{prop:vdecreasing}
	For each fixed $x>0$, the function $t\mapsto u(t,x)$ is decreasing.
\end{prop}
\begin{proof}
Let $M = \sup\limits_{u\in [0,1]}F'(x)$ and let $h,L\in(0,\infty)$.
If we define 
\[
z(t,x) = u^{[L]}(t+h,x)-u^{[L]}(t,x),
\]
then by
subtracting two copies of \cref{eq:uLPDE}, we see that
\begin{align*}
    \partial_t z(t,x) &= \frac12\Delta z(t,x) + \zeta[F(u^{[L]}(t+h,\gamma^{-1}x))-F(u^{[L]}(t,\gamma^{-1}x))-z(t,x)]\nonumber\\
                      &\le\frac12\Delta z(t,x) + \zeta[Mz(t,\gamma^{-1}x)-z(t,x)]
\end{align*}
and $z(t,0)=z(t,L)=0$ for all $t>0$. Furthermore, at $t=0$, we have 
\[
\hbox{$z(0,x) = u^{[L]}(t,x)-1\le  0$ for all $x\in (0,L)$.}
\]
		      Therefore, by the comparison principle (\cref{prop:comparisonprinciple}, with $G(u)=Mu$), we   
		      see that
		      \[
		      \hbox{$z(t,x)\le 0$ for all $t,x> 0$.}
		      \]
		      Since this is true for every $h>0$, this means that $t\mapsto u^{[L]}(t,x)$ is decreasing for each fixed $L$ and $x$. But since this is true for every $L$, and $u^{[L]}(t,x)\to u(t,x)$ as $L\to\infty$ as noted in \cref{eq:vLtov}, we see that $x\mapsto u(t,x)$ is decreasing for every fixed $x$ as well.
\end{proof}

\begin{prop}\label{prop:vconverges}
    There is a function $U_\infty:[0,\infty)\to[0,1]$ such that, for each $x>0$, we have
	\[
		\lim_{t\to\infty} u(t,x)=U_\infty(x).
	\]
\end{prop}
\begin{proof}
	We note that $\underline u(t,x) = 0$ is obviously a subsolution, satisfying \cref{eq:subsoln}. This means by \cref{prop:comparisonprinciple,prop:vdecreasing} that $u(t,x)$ is bounded below and decreasing in $t$, so 
	\[
	U_\infty(x)\coloneqq\lim\limits_{t\to\infty} u(t,x)
	\]
	exists and is nonnegative.
\end{proof}

Now we can prove \cref{thm:convergence}.
\begin{proof}[Proof of \cref{thm:convergence}]
    In light of \cref{prop:vconverges}, it remains to show that $U_\infty$ is the maximal solution of \crefrange{eq:nonlocalelliptic}{eq:nonlocalellipticlt1}.
    
    First we claim that $U_\infty$ is in fact a solution. Let $\|u(t,\cdot)\|_{\alpha}$ denote the $\clC^\alpha$ (Hölder) norm of $u(t,\cdot)$. Now we use parabolic regularity estimates such as those of \cite{KS81}, along with the fact that 
    $u$ is uniformly bounded, to obtain a $\delta>0$ such that
    $\|u(t,\cdot)\|_{\delta}$ is bounded uniformly in $t$.
    But since $F$ is $\clC^2$, this means that $\|F(u(t,\gamma^{-1}\cdot))\|_{\delta}$ is also bounded by a constant, uniformly in $t$. Schauder estimates (e.g.~\cite[Theorem~8.11.1]{Kry96}) then imply   %
    that~$\|u(t,\cdot)\|_{2+\delta}$ is uniformly bounded in $t$. In particular, this means that, for any compact $\clK\subseteq (0,\infty)$, the family $(u(t,\cdot)|_\clK)_t$ is compact in $\clC^2(\clK)$, so \begin{equation}u(t,\cdot)\to U_\infty(\cdot )\qquad\text{in } \clC^2(\clK).\label{eq:liminC2}\end{equation} Also, the family $(\frac12\Delta u(t,\cdot)+\zeta[F(u(t,\gamma^{-1}\cdot))-u(t,\cdot)]|_\clK)_t$ is precompact in $\clC^0$, so it converges to a limit in $\clC^0$. On the other hand, since there is a sequence $t_k\nearrow\infty$ such that 
    \[
    \lim\limits_{t\to\infty}\partial_t u(t_k,x)=0,
    \]
    and $u$ satisfies \cref{eq:uPDE}, this means that 
    \[
    \adjustlimits\lim_{k\to\infty} \sup_{x\in\clK}\Big|\frac12\Delta u(t_k,x)+\zeta[F(u(t_k,\gamma^{-1}x))-u(t_k,x)]\Big| =0.
    \] 
    Together with \cref{eq:liminC2}, this implies that 
    \[
    \frac12\Delta U_\infty + \zeta[F(U_\infty(\gamma^{-1}\cdot))-U_\infty]=0.
    \]

        Now we show that $U_\infty$ is the maximal solution to \crefrange{eq:uPDE}{eq:uIC}. Indeed, suppose there is another solution $V$ to \crefrange{eq:uPDE}{eq:uIC} and there is some $x>0$ such that $V(x)>U_\infty(x)$. But then \cref{prop:comparisonprinciple} implies that $u^{[L]}(t,x)\ge V(x)$ for all $t,L\in(0,\infty)$, so it is impossible that $u(t,x)\to U_\infty(x)$ as $t\to\infty$, which is a contradiction.
\end{proof}

\subsection{The supersolution}\label{sec:super}

In this subsection we prove the first part of \cref{thm:dichotomy}.
For $\delta,\xi\ge 0$ and $\omega\in (0,1]$, we define
\begin{equation}
    v_{\delta,\xi,\omega}(t,x) = \e^{-\delta t}\xi x^\omega.
\end{equation}
\begin{prop}
    Assume that %
\begin{equation}
    0\le\delta\le \zeta[1-\Upsilon_F\gamma^{-\omega}].
    \label{eq:UpsilonFltgamma}
\end{equation}
Then for all $\xi>0$,
$t\ge 0$, and $x\in (0,\infty)$ we have 
\begin{equation}\label{eq:vdeltasupersoln}
    \partial_tv_{\delta,\xi,\omega}(t,x)\ge\frac12\Delta v_{\delta,\xi}(t,x) + \zeta [ F(v_{\delta,\xi}(t,\gamma^{-1} x))- v_{\delta,\xi}(t,x)].
\end{equation}
\end{prop}
\begin{proof}
    Since $\omega\le 1$, we have $\Delta v_{\delta,\xi,\omega}(t,x) \le 0$ for all $x\ge 0$. Moreover, we have
    \[F(v_{\delta,\xi,\omega}(t,\gamma^{-1} x)) - v_{\delta,\xi}(t,x) \le \Upsilon_Fv_{\delta,\xi,\omega}(t,\gamma^{-1} x)-v_{\delta,\xi,\omega}(t,x) = \e^{-\delta t}\xi x^\omega[ \Upsilon_F\gamma^{-\omega} -1],\]
    whereas 
    \[\partial_tv_{\delta,\xi,\omega}(t,x) = -\delta\e^{-\delta t}\xi x^\omega.\]
    Thus 
    \cref{eq:vdeltasupersoln} follows %
    from \cref{eq:UpsilonFltgamma}.
\end{proof}

\begin{prop}
    \label{prop:lt0}
    Under the assumption \cref{eq:UpsilonFltgamma}, there exists a $\xi>0$ such that  
    \begin{equation}\label{dec1402}
    \hbox{$u(t,x) \le v_{\delta,\xi,\omega}(t,x)$ for each $t\ge 1$ and $x\in\RR$.}
    \end{equation}
\end{prop}
\begin{proof}
    Since $u$ can be extended to a solution to \cref{eq:uPDE-wholeline} on the whole line and $0<\omega\le 1$, it follows from standard properties of the heat equation (e.g. Schauder estimates as in the proof of \cref{thm:convergence}) that %
 there is a $\xi>0$ large enough that \begin{equation}v_{\delta,\xi,\omega}(1,x)\ge u(1,x)\qquad\text{for all $x\ge 0$.}\end{equation} Let \[\mu >\zeta[F'(0)-1]\] and define
    \[
    w(t,x) = \e^{-\mu t}(v_{\delta,\xi,\omega}-u)(t,x).\]
    Now fix $\eps>0$ and let \[T = \inf\left\{t\ge 1\ :\ \inf_{x\ge 0}w(t,x)\le -\eps\right\}.\] Suppose for the sake of contradiction that $T<\infty$. Since $u$ is bounded above by $1$, we note that $w(T,0) = 0$ and $\lim\limits_{x\to+\infty}w(T,x)=+\infty$. This means that $w(T,\cdot)$ achieves its minimum at some $X\in (0,\infty)$. We then derive a contradiction in the same way as in the proof of \cref{prop:comparisonprinciple}.
\end{proof}

\begin{proof}[Proof of \cref{thm:dichotomy}\textup{(\ref{enu:super})}]
    If $\gamma > \Upsilon_F$, then \cref{eq:UpsilonFltgamma} can be satisfied with $\omega =1$ and some $\delta>0$. Since with these choices we have  $\lim\limits_{t\to\infty}v_{\delta,\xi,\omega}(t,x)=0$ for each fixed $x\ge 0$, 
\cref{thm:dichotomy}(\ref{enu:super}) follows from \cref{prop:vconverges,prop:lt0}.
\end{proof}

\begin{proof}[Proof of \cref{thm:speed}\textup{(\ref{enu:speedsuper})}]
    Take $\omega$ to achieve the supremum in \cref{eq:Sigmadef}, with $f\mapsfrom \Upsilon_F$, and let 
    \[
    \delta = \zeta[1-\Upsilon_F\gamma^{-\omega}] = \omega \zeta\Sigma(\Upsilon_F,\gamma).
    \]
    Then \cref{prop:lt0} tells us that there is a $\xi>0$ such that
    \[
    u(t,x)\le v_{\delta,\xi,\omega}(t,x) = v_{0,\xi,\omega}(1,\e^{-\zeta\Sigma(\Upsilon_F,\gamma)t}x)\qquad\text{for all $t\ge 1$ and $x\ge 0$,}
    \]
    which is  \cref{eq:speed-super}.
\end{proof}

\subsection{The subsolution}\label{sec:sub}

In this subsection we prove the \ordinalstringnum{\getrefnumber{enu:sub}} parts of \cref{thm:dichotomy,thm:speed}. Throughout this section, we fix $F$, $\zeta$, and $\gamma$ and also fix a \begin{equation}\label{eq:nucond}\nu >\zeta\Sigma(F'(0),\gamma)\end{equation} (recalling the definition \cref{eq:Sigmadef}
and the assumption that $F'(0)>1$).
We also recall the definition of $\Xi_F$ in the statement of \cref{thm:dichotomy}.

We start by constructing an auxiliary function $H$, which will operate as a stand-in for~$F$ with some additional useful properties. The following lemma is proved in Section~\ref{sec:H-constr} below.
\begin{lem}\label{lem:HfromF}
    For any $f\in (1,F'(0))$, there is a $\delta\in (0,\frac12\Xi_F)$ and a $\clC^2$ function 
    \[
    H\colon[0,\Xi_F]\to[0,\Xi_F]
    \]
    with the following properties:
    \begin{enumerate}
        \item $H(0) =0$, $H(\Xi_F)=\Xi_F$, and %
            \begin{equation}
                \label{eq:HbelowF}
                u<H(u)\le F(u) %
                \qquad\text{for all }u\in (0,\Xi_F).
            \end{equation}
        \item $H'(u)>0$ for all $u\in [0,\Xi_F]$. In particular, $H'(\Xi_F)>0$.
         \label{enu:H'pos}
        \item $H(x) = fx$ for all $x\in [0,\delta]$.\label{enu:Hstartslinear}
        \item $H$ is convex on $[\Xi_F-\delta,\Xi_F]$.\label{enu:Hconvex}
    \end{enumerate}
\end{lem}

For the rest of this section, we fix $f\in(1,F'(0))$ so close to $F'(0)$, and \begin{equation}B\in(0,\gamma^{-1})\label{eq:Bdef}\end{equation} 
so close to $\gamma^{-1}$, that the weakening \begin{equation}\nu>\zeta\Sigma(f,B^{-1})\label{eq:nucondm}\end{equation} of \cref{eq:nucond} holds. 
With these choices, we also fix $\delta\in (0,\frac12\Xi_F)$ and the function $H$ from \cref{lem:HfromF}.

The next lemma is the key to building the subsolution.

\begin{lem}\label{lem:startthesubsolution}     
There is a function $w:[0,\infty)\to [0,\Xi_F]$ with the following properties:
    \begin{enumerate}
        \item The function $w$ is strictly increasing, continuous, and piecewise $\clC^2$. At each point $x$ at which $w$ is not $\clC^2$, it satisfies $w'(x-)< w'(x+)$.
        \item The function $w$ is linear on a nonempty interval starting at $0$.\label{enu:starts-linear}
        \item We have \begin{equation}w(x) = H(w(Bx))\qquad\text{for each }x\ge B^{-1}.\label{eq:wrecursion}\end{equation}
        \item For each $x\in (0,\infty)$ at which $w$ is differentiable, we have
            \begin{equation}\label{eq:wissubsol}
                0\le \zeta[H(w(Bx))-w(x)]+\nu x w'(x).
            \end{equation}
    \end{enumerate}
\end{lem}
\begin{rem}
Before we prove \cref{lem:startthesubsolution}, we note two immediate consequences:
\begin{enumerate}\item For every $x\in (0,\infty)$ at which $w$ is differentiable, we have
\begin{equation}
\begin{aligned}
    \zeta[F(w(\gamma^{-1}x))-w(x)] + \nu x w'(x) &\ge \zeta[H(w(\gamma^{-1}x))-w(x)]+\nu x w'(x)\\
                                                                     &%
                                                                     > \zeta[H(w(Bx))-w(x)]+\nu x w'(x)\ge0,
\end{aligned}
\label{eq:therealinequality}
\end{equation}
with the first inequality holding by \cref{eq:HbelowF} and the second because $H$ and $w$ are strictly increasing and $\gamma^{-1}>B$ by the choice of $B$.
\item Since $H$ is increasing on $(0,\Xi_F)$ and $H(u)>u$ for all $u\in (0,\Xi_F)$ while $H(\Xi_F)=\Xi_F)$, %
    \cref{eq:wrecursion} implies that %
    \begin{equation}\label{eq:wlimitatinfty}
        \lim_{x\to+\infty} w(x) = \Xi_F.
    \end{equation}
\end{enumerate}

\end{rem}

\begin{proof}[Proof of \cref{lem:startthesubsolution}.]
As $f>1$ and $0<B<\gamma^{-1}<1$, we have
    \begin{equation}\omega_0: = -\log_B f=-\frac{\ln f}{\ln B}>0. \label{eq:omega0def}
    \end{equation} 
We can apply \cref{lem:build-sub-inductive} below, with $\nu\mapsfrom \nu/\zeta$.
In that case, the assumption~\cref{eq:kappagt0} in that lemma takes the form
\begin{equation}\label{dec1912}
\inf_{\omega\in(0,1]}\Big(fB^{\omega}-1+\frac{\nu}{\zeta}\omega\Big)>0. 
\end{equation}
It is satisfied by~\cref{eq:nucondm} since for any $\omega\in[0,1]$ we have
\begin{equation}\label{dec1914}
fB^{\omega}-1+\frac{\nu}{\zeta}\omega>fB^{\omega}-1+\Sigma(f,B^{-1})\omega\ge 0
\end{equation}
by the definition of $\Sigma(f,B^{-1})$.
The aforementioned lemma, for each $\omega\in(0,1]$ fixed, gives us a continuous
function $v_\omega\colon [0,1]\to [0,1]$ such that $v_\omega$ is linear in a neighborhood of~$0$, 
is piecewise $\clC^2$, and has an upward jump of the derivative 
at every point where $w_\omega$ is not~$\clC^2$. Moreover, we have both
    \begin{equation}\label{eq:vomegaend}
        v_\omega(x) = \frac{m_\omega}{\omega}(x^\omega-1)+1\qquad\text{for all }x\in [B,1],
    \end{equation}
with $m_\omega\in\RR$ as in \cref{eq:momegacond},    and 
\begin{equation}\label{dec1806}
0\le \zeta[f v_\omega(Bx)-v_\omega(x)]+\nu xv'_\omega(x)\qquad\text{for all }x\in [0,1].
\end{equation}
Choosing $\alpha\in(0,1)$ %
in the statement of \cref{lem:build-sub-inductive} 
sufficiently close to $1$ and taking $\omega<\alpha\omega_0$, 
we can furthermore use \cref{eq:momegacond} to ensure that
    \begin{equation}\label{eq:readytoglue}
        m_\omega <\alpha^{-1}\omega< \omega_0. %
    \end{equation}
    Moreover, we have %
    \begin{equation}\label{eq:condglues}
        x^{\omega_0}\le \frac{m_\omega}{\omega}(x^\omega-1)+1\overset{\cref{eq:vomegaend}}=v_\omega(x)\qquad\text{for all }x\in [B,1].
    \end{equation}
    Now, we recall $\delta\in(0,\frac{1}{2}\Xi_F)$ from \cref{lem:HfromF}, and define
    \[
        w(x)=\begin{cases}
            B^{\omega_0}\delta v_\omega(x),&x\in [0,1];\\
            B^{\omega_0}\delta x^{\omega_0},&x\in [1,B^{-1}].\\
        \end{cases}
    \]
We extend $w$ to a function on $[0,\infty)$ by inductively imposing \cref{eq:wrecursion}.

We note that  
\[
w(1-)=B^{\omega_0}\delta v_\omega(1) = B^{\omega_0}\delta = w(1+),
\]
and
\[
w'(1-)= B^{\omega_0}\delta v_\omega'(1) = 
B^{\omega_0}\delta m_\omega \overset{\cref{eq:readytoglue}}
< B^{\omega_0}\delta\omega_0 = w'(1+).
\] %
Moreover, for $x\in [B^{-1},B^{-2}]$, we have
\[
w(Bx)=B^{\omega_0}
\delta(Bx)^{\omega_0}\le B^{\omega_0}\delta(B^{-1})^{\omega_0}=\delta.
\]
Thus,~\cref{enu:Hstartslinear} in \cref{lem:HfromF} implies that for $x\in [B^{-1},B^{-2}]$
\begin{equation}\label{dec1804}
 w(x) = H(w(Bx)) = f B^{\omega_0}\delta (Bx)^{\omega_0}
     = B^{\omega_0}\delta x^{\omega_0},
 \end{equation}
by the definition \cref{eq:omega0def}
of $\omega_0$. In particular, $w$ is smooth at $x=B^{-1}$.
Since $H$ is $\clC^2$, this, together with the extension formula~\cref{eq:wrecursion},
implies that $w$ is $\clC^2$ on $(1,\infty)$. Therefore,~$w$ is piecewise $\clC^2$ and its derivative has an upward jump at every point at which it is not~$\clC^2$.

Let us also note that 
\begin{equation}\label{dec1808}
w(Bx)\le \delta B^{\omega_0}\le\delta,~~\hbox{for all $x\in[0,B^{-1}]$},
\end{equation}
from which we deduce that
\begin{equation}\label{dec1810}
H(w(Bx))=fw(Bx),~~\hbox{for all $x\in[0,B^{-1}]$},
\end{equation}

Now, first, \cref{eq:wissubsol} for $x\ge B^{-1}$ is clear from the definition \cref{eq:wrecursion}
of $w$ %
since $w(x)$ is increasing.
It is also true for $x\le 1$ by~\cref{dec1806,dec1810}. 
Finally, \cref{eq:wissubsol} also holds for
$x\in [1,B^{-1}]$ since, by  
\cref{eq:omega0def,eq:vomegaend,eq:condglues,dec1810}, 
we have for such $x$ that %
\begin{align*}
\zeta[H(w(Bx))-w(x)]+\nu x w'(x)&=
\zeta[f w(Bx) - w(x)]+\nu x w'(x) \\
&\ge \zeta[fB^{\omega_0}\delta (Bx)^{\omega_0}-B^{\omega_0}\delta x^{\omega_0}]
=0. 
\qedhere
\end{align*} %
\end{proof}

Next, we prove the following estimate on the oscillation of $w$ on dyadic intervals 
far from the origin.
\begin{lem}\label{lem:lousy-harnack}
    For any $\kappa>0$ and $A\in(1,\infty)$, we have constants $C,M\in(0,\infty)$ such that, whenever $x,y\ge M$ are such that $|\log x-\log y|\le \log A$, we have
    \begin{equation}
        \label{eq:wprimesclose}
        |\log w'(y)-\log w'(x)|\le C+\kappa\log y.
    \end{equation}
\end{lem}
\begin{proof}
    Fix $\alpha>0$, to be chosen later.
Since $H'(\Xi_F)>0$ by \cref{lem:HfromF}(\ref{enu:H'pos}) and $\lim\limits_{x\to\infty} w(x)=\Xi_F$ by \cref{eq:wlimitatinfty}, if $M\in (0,\infty)$ is sufficiently large, then we have
    \begin{equation}\label{dec1812}
        \frac{H'(w(y))}{H'(w(x))}\le B^{-\alpha}\qquad\text{for all }x,y\ge M.
    \end{equation}
Differentiating \cref{eq:wrecursion} and using~\cref{dec1812}, 
we obtain, whenever $x,y\ge M$, that
    \begin{equation}\label{dec1814}
        \frac{w'(B^{-1}y)}{w'(B^{-1}x)}=\frac{H'(w(y))w'(y)}{H'(w(x))w'(x)}\le B^{-\alpha}\frac{w'(y)}{w'(x)}.
    \end{equation}
    Now let \[\xi = \sup_{x,y\in [M,B^{-1}M]}\frac{w'(y)}{w'(x)}.\]
Using induction in~\cref{dec1814}
implies that, for each $k\in\NN$, we have
    \[
        \frac{w'(B^{-k}y)}{w'(B^{-k}x)}\le B^{-k\alpha}\xi\qquad\text{for all }x,y\in [M,B^{-1}M],
    \]
    which means that
    \[
        \frac{w'(y)}{w'(x)}\le B^{-k\alpha}\xi\le \left(\frac{y}{BM}\right)^{\alpha}\xi\qquad\text{for all }x,y\in [B^{-k}M,B^{-k-1}M].
    \]
    By the triangle inequality, this implies that
    \[
        \frac{w'(y)}{w'(x)}\le \left(\frac{y}{BM}\right)^{2\alpha}\xi^2\qquad\text{for all }x,y>M\text{ such that }|\log x-\log y|\le \log B^{-1},
    \]
    and hence that, for any $N\in\NN$,
    \[
        \frac{w'(y)}{w'(x)}\le \left(\frac{y}{BM}\right)^{2N\alpha}\xi^2\qquad\text{for all }x,y>M\text{ such that }|\log x-\log y|\le \log B^{-N}.
    \]
    Choosing $N$ large enough that $B^{-N}\ge A$, and then $\alpha=\kappa/(2N)$, this gives us \cref{eq:wprimesclose}.
\end{proof}

The next lemma proves that the subsolution property holds far from the origin.
\begin{lem}\label{lem:large-x}
    There is an $M<\infty$ such that
    \begin{equation}
        0\le \frac12\Delta w(x)+ \zeta[F(w(\gamma^{-1}x))-w(x)]\qquad\text{for all }x\ge M.\label{eq:w-goal}
    \end{equation}
\end{lem}
\begin{proof}
    First we estimate that
    \begin{equation}
    \begin{aligned}
  F(w(\gamma^{-1} x))-w(x) &\ge H(w(\gamma^{-1}x))-w(x))
        =w( x/(\gamma B))-w(x)
                                      \\
                                      &\ge (1/(B\gamma)-1)x\inf_{[x, x/(\gamma B)]} w'
                                      \ge \frac{x^{0.9}}Cw'(x)\label{eq:nonlocalterm}
\end{aligned}
\end{equation}
for a sufficiently large constant $C$ as long as $x$ is sufficiently large, where the first inequality is by \cref{eq:HbelowF}, the equality is by \cref{eq:wrecursion}, %
the second inequality
is by the mean value theorem, and the third is by \cref{lem:lousy-harnack,eq:Bdef}.

Next, we differentiate \cref{eq:wrecursion} to obtain that, for all $x\ge B^{-1}$, we have
    \begin{equation}
        w'(x) = BH'(w(Bx))w'(Bx).\label{eq:w'recursion}
    \end{equation}
    and
    \begin{equation}
        w''(x) = B^2\left[H''(w(Bx))(w'(Bx))^2 + H'(w(Bx))w''(Bx)\right].\label{eq:w''recursive}
    \end{equation}
    If $x$ is large enough that $w(x)\ge \Xi_F-\delta$, then by \cref{eq:w''recursive} 
    and \cref{lem:HfromF}(\ref{enu:Hconvex}), we have \[w''(B^{-1}x) = B^2\left[H''(w(x))(w'(x))^2 + H'(w(x))w''(x)\right]\ge B^2H'(w(x))w''(x).\] Dividing this inequality by \cref{eq:w'recursion}, we obtain
\[
    \frac{w''(B^{-1}x)}{w'(B^{-1}x)}\ge \frac{Bw''(x)}{w'(x)}.
\]
By induction, this means that, for each $k\ge 1$, we have
\[
\frac{w''(B^{-k}x)}{w'(B^{-k}x)}\ge \frac{B^kw''(x)}{w'(x)}.
\]
Therefore, increasing the constant $C$ if necessary, we obtain %
\begin{equation}w''(x)\ge-\frac C x w'(x)\qquad\text{for all $x$ sufficiently large}.\label{eq:laplacianterm}\end{equation}
Now combining \cref{eq:nonlocalterm} and \cref{eq:laplacianterm}, we have for sufficiently large $x$ that 
\[w''(x)+\zeta[F(w(\gamma^{-1} x))-w(x)]\ge \left(-\frac C x  + \frac{\zeta x^{0.9}}{C}\right)w'(x).\]
For large enough $x$, the right side is positive, so the proof is complete.
\end{proof}
Now we can complete the proofs.
\begin{proof}[Proof of \cref{thm:dichotomy}\textup{(\ref{enu:sub})} and \cref{thm:speed}\textup{(\ref{enu:speedsub})}]
    Let $\underline v(t,x) = w(\e^{-\nu t}\xi x)$ for some $\xi > 0$ to be chosen. Then $v(t,\cdot)$ is piecewise $\clC^2$ for every $t$ and \cref{eq:derivcond} is satisfied. Moreover, we can compute at every point of differentiability that
    \begin{equation}
        \begin{aligned} \frac12&\Delta \underline v(t,x)+\zeta[F(\underline v(\gamma^{-1} x))-\underline v(x)] -\partial_t \underline v(t,x)\\&= \frac12\xi^2\e^{2\nu t} \Delta w(y) + \zeta[F(w(\gamma^{-1} y))-w(y)]+\nu yw'(y),\end{aligned}\label{eq:timederivbecomesspacederiv}
\end{equation}
    where we have defined $y=\e^{-\nu t}\xi x$. %
    Therefore, by \cref{lem:startthesubsolution}(\ref{enu:starts-linear}), \cref{eq:therealinequality}, and \cref{lem:large-x}, there is an $M\in(1,\infty)$ such that if $y<M^{-1}$ or $y>M$, then for any $\xi\le 1$, we have \begin{equation}\label{eq:whatweneed}\frac12\xi^2 \Delta w(y)+\zeta[F(w(\gamma^{-1} y))-w(y)]+\nu y w'(y)\ge 0.\end{equation} Also, \cref{eq:therealinequality} tells us that, by taking $\xi$ sufficiently close to $0$, we can ensure that \cref{eq:whatweneed} holds for $y\in [M^{-1},M]$ as well. Combined with \cref{eq:timederivbecomesspacederiv}, this means that
    \[\frac12\Delta \underline v(t,x)+\zeta[F(\underline v(\gamma^{-1} x))-\underline v(x)] -\partial_t \underline v(t,x)\ge 0\] at every point $(t,x)$ of differentiability. By the comparison principle \cref{prop:comparisonprinciple}, 
and keeping in mind the initial conditions~\cref{eq:uLic}
and the boundary conditions~\cref{eq:uLleft,eq:uLright} for~$u^{(L)}$,
we therefore have $u^{[L]}(t,x)\ge \underline v(t,x)$ 
    for all $t,x,L\ge 0$, and hence by \cref{eq:vLtov} we have 
    \[
    u(t,x)\ge \underline v(t,x) = w(\e^{-\nu t}\xi x),\hbox{ for all $t,x\ge 0$.}
    \]
    This implies \cref{eq:speed-sub}, and if $\gamma<F'(0)$, then we can take $\nu =0$ and recall \cref{eq:limitexists} to obtain~\cref{eq:Uinftynottozero}.
\end{proof}

\subsection{Analytic constructions}\label{sec:analyt}

In this section, we prove~\cref{lem:HfromF}, building the function $H(u)$. 
We also prove the auxiliary~\cref{lem:build-sub-inductive}
that was used to construct the function $v_\omega(x)$ used throughout the construction of the
sub-solution.

\subsubsection{Construction of \texorpdfstring{$H$}{H}}\label{sec:H-constr}

\begin{lem}\label{lem:buildH}
    Let $b>0$.
    Suppose that $G\in \clC^1([0,b])$ is such that $G(0) = 0$, $G(b) = b$, $G'(0)>1$, 
    and $u< G(u)\le b$ for all $u\in [0,b]$.
Then for any $g\in (1,G'(0))$, there is a function $H\in\clC^2([0,1])$ and $\delta\in(0,b)$
with the following properties:
\begin{enumerate}
    \item $H(0)=0$, $H(b)=b$, and 
        \begin{equation}
            u<H(u)\le G(u)\qquad\text{for all }u\in (0,b).\label{eq:HbelowG}
        \end{equation}
    \item $H'(u)>0$ for all $u\in [0,b]$.
    \item $H(u)=gu$ for all $x\in [0,\delta]$.
    \item There is a $\delta>0$ such that $H$ is convex on $[b-\delta,b]$.
\end{enumerate}
\end{lem}

To prove this, we first need the following lemma:
\begin{lem}\label{lem:cvxbelow}
    Suppose that $G\colon[0,1/2]\to [0,\infty)$ is continuous and such that $G(0)=0$ 
    and $G(u)>0$ for all $u\in(0,1/2]$. Then there is a $\clC^2$ convex function $h\colon[0,1/2]\to [0,\infty)$ such that $h(0)=0$ and $0<h(u)< G(u)$ for all $u\in (0,1/2]$.
\end{lem}
\begin{proof}
Define 
\[
    k(x) = \min_{y\in [x,1/2]}G(y).
\]
Then it is clear that $k$ is increasing and continuous since $G$ is continuous, and moreover that for all $x\in (0,1/2]$ we have $k(x)>0$ since $G(x)>0$ for all $x\in (0,1/2]$. It is also clear from the definition that $k(x)\le G(x)$ for all $x\in [0,1/2]$. We now define
\[h(x) = \int_0^x \int_0^y k(z)\,\dif z\,\dif y.\]
By the fundamental theorem of calculus, since $k$ is nonnegative and continuous, we see that $h''(x)=k(x)$ and so $h$ is convex and $\clC^2$. If $x\in (0,1/2]$, then it is also clear that $h(x)>0$ since this property is true of $k$. Finally, we see that $h(0)=0$, and also that, for all $x\in (0,1/2)$,
\[h(x) \le \frac12x^2 \max_{y\in [0,x]}k(y)= \frac12 x^2 k(x)\le\frac18 G(x)<G(x).\]
Thus $h$ has all of the desired properties.
\end{proof}

Now we can prove \cref{lem:buildH}.

\begin{proof}[Proof of \cref{lem:buildH}]
Assume without loss of generality that $b=1$.
As $G(1)=1$, we may apply \cref{lem:cvxbelow} to the function
\[
\tilde G(u)=G(1-u)-(1-u).
\]
This gives a $\clC^2$ convex function $\tilde H_1(u)$ such that 
\[
\hbox{$\tilde H_1(0)=0$ 
and $0<\tilde H_1(u)<\tilde G(u)$ for all $0<u\le 1/2$.}
\]
Then, the function 
\[
H_1(u)=\tilde H_1(1-u)+u
\]
is also $\clC^2$ and convex and satisfies
\[
\hbox{$u<H_1(u)< G(u)$ for all $u\in [1/2,1)$,}
\]
and $H_1(1)=1$. Since $G(u)\le 1$ for all $u<1$, these properties imply that $H_1'(1)>0$, 
so there is a $\delta_1\in(0,1/2)$ such that 
\[
\hbox{$H_1'(u)>0$ for all $u\in[1-\delta_1,1]$.}
\]
Since $1<g<G'(0)$, we can also find a $\delta_2\in(0,1/2)$ such that 
\[
gu< G(u)\text{ for 
all $u\in (0,\delta_2]$.}
\]
Now we define $\delta_3 = \delta_2/g<\delta_2$. 
Then, we have 
\[
G(u)>gu\ge g\delta_3~~\hbox{for all $u\in [\delta_3,\delta_2]$,}
\]
as well as
\[
G(u)>u\ge \delta_2=g\delta_3~~\hbox{for all $u\in [\delta_2,1]$,}
\]
We see that
\[
\hbox{$G(u)>g\delta_3$ for all $u\in[\delta_3,1]$.}
\]
We also pick $\delta_4\in (0,\delta_1)$ small enough that $H_1(1-\delta_4)>g\delta_3$.

Now we let 
\[
\hbox{$H(u) = gu$ for $u\in [0,\delta_3]$,}
\]
and
\[
\hbox{ $H(u) = H_1(u)$ for $u\in [1-\delta_1,1]$.}
\]
It is then straightforward to extend $H$ on the interval
$(\delta_3,1-\delta_1)$ such that $H$ is $\clC^2$, $H'>0$, and $u<H(u)<G(u)$ for all $u\in (1-\delta_3,\delta_1)$. (See \cref{fig:HfromG}.)
	Then the proof is complete with $\delta=\delta_1\wedge \delta_3$.
        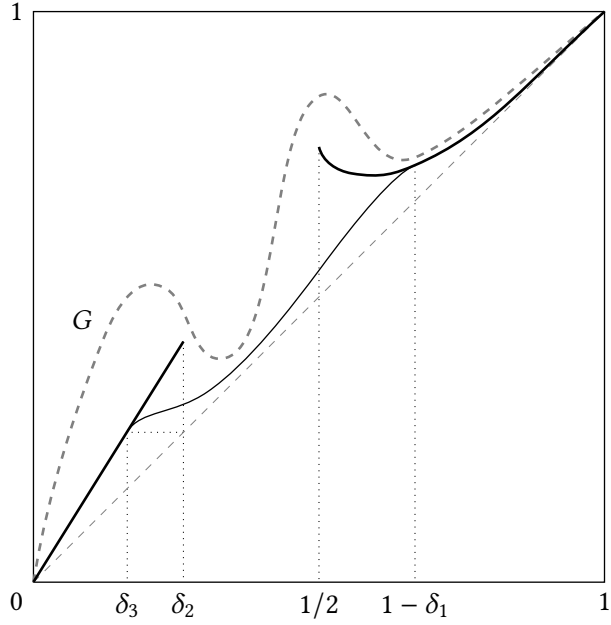
\begin{figure}
            \centering
            \def \globalscale {1.500000}
\begin{tikzpicture}[y=1pt, x=1pt, yscale=\globalscale,xscale=\globalscale]
  \draw[dotted] (23.65,38.02) -- (23.65,0.0) node [below] {$\delta_3$};
  \draw[dotted] (72, 109.845) -- (72,0.0) node [below] {$1/2$};
  \draw[dotted] (96.213,105.243) -- (96.213,0) node [below] {$1-\delta_1$};
  \draw[dotted] (37.817,60.723) -- (37.817,0) node [below] {$\delta_2$};
  \draw[dotted] (37.817,37.817) -- (23.65,37.817);
  \node [below left] at (0,0) {$0$};
  \node [below] at (144,0) {$1$};
  \node [left] at (0,144) {$1$};

  \path[draw=black,line width=0.5pt] (-0.0, 144.0) rectangle (144.0, 0.0);

  \path[draw=gray,dashed,line width=0.3pt] (0.0, 0.0) -- (144.0, 144.0);

  \path[draw=gray,dashed,line width=1.0pt] (0.0, 0.0).. controls (0.444, 1.767) and
  (3.193, 25.33) .. (17.621, 61.331).. controls (21.355, 70.646) and (26.57,
  78.124) .. (33.555, 73.999).. controls (39.541, 70.463) and (39.928, 54.073)
  .. (48.677, 56.651).. controls (60.801, 60.223) and (60.344, 109.527) ..
  (69.644, 120.572).. controls (79.219, 131.942) and (83.064, 102.081) ..
  (96.007, 107.109).. controls (108.905, 112.12) and (128.07, 131.051) ..
  (144.0, 144.0);
\node [above left] at (17.621,61.331) {$G$};

  \path[draw=black,line width=1pt] (0.0, 0.057) --
  (23.665, 38.02) -- (37.817, 60.723);

  \path[draw=black,line width=0.3pt,line cap=round,line join=round,line width=0.5pt] (23.665, 38.02).. controls (27.042, 43.437) and
  (35.549, 42.328) .. (43.804, 47.988).. controls (54.24, 55.143) and (64.962,
  69.297) .. (72.044, 78.747).. controls (82.0, 92.336) and (91.047, 103.028) ..
  (96.21, 105.242);

  \path[draw=black,line width=1pt] (144.0, 144.0)..
  controls (125.321, 127.061) and (115.063, 113.325) .. (96.213, 105.243)..
  controls (91.266, 103.122) and (87.813, 101.877) .. (80.198, 103.171)..
  controls (76.452, 103.802) and (72.876, 106.162) .. (72, 109.845);
\end{tikzpicture}
            \caption{Construction of $H$. The functions $u\mapsto gu$ and $H_1$ are plotted as thick black lines, and $H$ between them is plotted as a thin black line.\label{fig:HfromG}}
        \end{figure}
\end{proof}

\subsubsection{Construction of \texorpdfstring{$v_\omega$}{v\_omega}}

\begin{lem}
For any $x>0$, the function $\omega\mapsto\frac{x^{\omega}-1}{\omega}$
is increasing on $[0,\infty)$.\label{lem:omegathingincreasing}
\end{lem}

\begin{proof}
For $\omega_{2}>\omega_{1}>0$, we have
\begin{equation}\label{dec1904}
\frac{x^{\omega_{2}}-1}{\omega_{2}}=\frac{\omega_{2}-\omega_{1}}{\omega_{2}}\left(\frac{x^{\omega_{2}}-x^{\omega_{1}}}{\omega_{2}-\omega_{1}}\right)+\frac{\omega_{1}}{\omega_{2}}\left(\frac{x^{\omega_{1}}-1}{\omega_{1}}\right).
\end{equation}
By the convexity of $\omega\mapsto x^{\omega}$ and the mean value
theorem, we see that 
\[
\frac{x^{\omega_{2}}-x^{\omega_{1}}}{\omega_{2}-\omega_{1}}>\frac{x^{\omega_{1}}-1}{\omega_{1}},
\]
which, together with \cref{dec1904}, implies that 
\[
\frac{x^{\omega_{2}}-1}{\omega_{2}}>\frac{x^{\omega_{1}}-1}{\omega_{1}}.\qedhere
\]
\end{proof}
We now come to the crucial lemma that we used the construction of the subsolution.
\begin{lem}\label{lem:build-sub-inductive}
Suppose that $0<B<1<f$ and $\nu\ge0$ are such that
\begin{equation}
\kappa\coloneqq\inf_{\omega\in(0,1]}\left(fB^{\omega}-1+\nu\omega\right)>0.\label{eq:kappagt0}
\end{equation}
Fix also an $\alpha\in (0,1)$.
Then for each $\omega\in(0,1]$, there is a function
$v_{\omega}\colon[0,1]\to[0,1]$ with the %
following properties:
\begin{enumerate}
    \item \label{enu:firstprop}$v_\omega$ is increasing, continuous, and piecewise $\clC^2$. At any point $x$ at which $v_\omega$ is not $\clC^2$, we have \begin{equation}v_\omega'(x-)<v_\omega'(x+).\label{eq:deriv-ordering}\end{equation}
\item There is a $\delta=\delta_\omega>0$ such that the restriction $v_{\omega}|_{[0,\delta_{\omega}]}$ is linear.
\item We have an 
\begin{equation}
m_{\omega}\in[\omega,\alpha^{-1}\omega)\label{eq:momegacond}
\end{equation}
 such that
\begin{equation}
v_{\omega}(x)=\frac{m_{\omega}}{\omega}(x^{\omega}-1)+1\qquad\text{for all }x\in[B,1].\label{eq:vomegalargeval}
\end{equation}
\item \label{enu:lastprop} We have
\begin{equation}
0\le fv_{\omega}(Bx)-v_{\omega}(x)+\nu xv_{\omega}'(x)\qquad\text{for all }x\in[0,1].\label{eq:fvBomega}
\end{equation}
\end{enumerate}
\end{lem}

\begin{proof}
As decreasing $\alpha$ makes \cref{eq:momegacond} easier to satisfy,
we may assume without loss of generality that   
    \begin{equation}\label{dec1802}
\alpha>\left(1-(f-1)^{-1}\kappa\right)_{+}^{1/3}.
\end{equation}
The parameter $\alpha$ will remain fixed throughout the remainder of the proof.

Let $\mathcal{K}$ be the set of $\omega\in (0,1]$ such that there exists a function $v_{\omega}\colon [0,1]\to [0,1]$ satisfying properties~\ref{enu:firstprop}--\ref{enu:lastprop}.  We will show that $\mathcal{K}=(0,1]$ by an inductive argument.
First we show that $1\in\mathcal{K}$. Indeed, this is clear since we
can take $v_{1}(x)=x$ for all $x\in[0,1]$, and it is straightforward
to check that the required properties are all satisfied (with $\delta_{1}=m_{1}=1$).
In particular,~\cref{eq:fvBomega} holds in that case because of \cref{eq:kappagt0}.

Now suppose for the sake of contradiction that $(0,1]\setminus\mathcal{K}$
is nonempty, let
\[
\omega_{0}=\sup[(0,1]\setminus\mathcal{K}],
\]
and take 
\[
\omega_{1}\in\mathcal{K}\cap[\omega_{0},\alpha^{-1}\omega_{0})
\]
and   
\begin{equation}\label{eq:omega2def}
\omega_2\in [\alpha\omega_0,\omega_0]\setminus\mathcal{K}.
\end{equation}
Then 
\begin{equation}\label{eq:omega12ordering}
    \alpha^2\omega_1<\omega_2<\omega_1,
\end{equation}
and so \cref{eq:momegacond} implies that
\begin{equation}
    \label{eq:slopecond}
\frac{m_{\omega_{1}}}{\omega_{2}}<\alpha^{-1}\frac{\omega_1}{\omega_2}\le\alpha^{-3}.
\end{equation}
Therefore, we can fix
\begin{equation}\label{eq:mchoice}m\in (m_{\omega_1},\alpha^{-3}\omega_2).\end{equation} %

We now 
define
\[
\tilde{v}(x)=\begin{cases}
v_{\omega_{1}}(x), & x\in[0,1];\\
\dfrac{m}{\omega_{2}}(x^{\omega_{2}}-1)+1, & x\ge1.
\end{cases}
\]
	(See \cref{fig:vomegaplot}.)
    \begin{figure}
        \centering
            \begin{tikzpicture}
        \begin{axis}[xlabel = \(x\),ylabel={\(v_{\omega_2}(x)\)}, width=2in]
    \addplot[domain=0:1,samples=2]{x};
    \addplot[domain=1:3,samples=30]{2/0.2*(x^0.2-1)+1};
\end{axis}
\end{tikzpicture}
\caption{Possible plot of $\tilde v$, with the choices $\omega_2=0.2$, $m=2$, and $\omega_1=m_1=1$.\label{fig:vomegaplot}}
\end{figure}
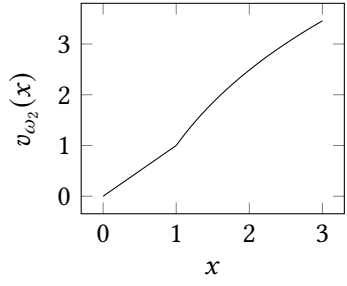
Since $v_{\omega_{1}}(1)=1$, we see that $\tilde v$ is continuous. Moreover, since 
\[
v_{\omega_{1}}'(1)=m_{\omega_{1}}<m,
\]
we see (using \cref{eq:deriv-ordering} since $\omega_1\in\clK$) that $\tilde v'(x-)<\tilde v '(x+)$ at every $x$ at which $\tilde v$ is not $\clC^2$. %
It is also clear that $\tilde v$ is increasing.

We now observe that
\begin{equation}
    \tilde{v}(x)\ge\frac{m}{\omega_{2}}(x^{\omega_{2}}-1)+1\qquad\text{for all }x\ge B.\label{eq:tildevgreater}
\end{equation}
Indeed, equality holds for $x\ge 1$ by definition. On the other hand, for $x\in [B,1]$
we have $x^{\omega_2}-1\le 0$, while  
$m>m_{\omega_1}$. 
It follows from \cref{eq:vomegalargeval,lem:omegathingincreasing} that 
\[
\frac{m}{\omega_{2}}(x^{\omega_{2}}-1)+1\le
\frac{m_{\omega_1}}{\omega_2}(x^{\omega_2}-1)+1 \le \frac{m_{\omega_1}}{\omega_1}(x^{\omega_1}-1)+1 = \tilde v(x),~~\hbox{ for $x\in[B,1]$.} 
\]
We can then conclude, for all $x\ge1$, that
\begin{align}
f & \tilde{v}(Bx)-\tilde{v}(x)+\nu x\tilde{v}'(x)\overset{\cref{eq:tildevgreater}}\ge f\left(\frac{m}{\omega_{2}}(B^{\omega_{2}}x^{\omega_{2}}-1)+1\right)-\left(\frac{m}{\omega_{2}}(x^{\omega_{2}}-1)+1\right)+\nu mx^{\omega_{2}}\nonumber \\
 & =\frac{m}{\omega_{2}}\left(fB^{\omega_{2}}-1+\nu \omega_{2}\right)x^{\omega_{2}}-(f-1)\left(\frac{m}{\omega_{2}}-1\right)\overset{\cref{eq:kappagt0}}{\ge}\frac{m}{\omega_{2}}\kappa-(f-1)\left(\frac{m}{\omega_{2}}-1\right)\nonumber \\
 & =\frac{m}{\omega_{2}}\left(\kappa-(f-1)\right)+f-1.\label{eq:fBxk}
\end{align}
Now if $\kappa\ge f-1$, then the last term is positive since $\omega_{2},m,f-1>0$.
On the other hand, if $\kappa<f-1$, then, recalling  \cref{dec1802,eq:mchoice},
we see that
\[
m<\alpha^{-3}\omega_2\le (1-(f-1)^{-1}\kappa)^{-1}\omega_2=\frac{f-1}{(f-1)-\kappa}\omega_2
\]
Thus, in that case we have  %
\[
\frac{m}{\omega_{2}}\left(\kappa-(f-1)\right)>-(f-1).
\]
Hence, the right side of \cref{eq:fBxk} is positive in this
case as well. 

Combining \cref{eq:fBxk} with the hypothesis \cref{eq:fvBomega}
for $\omega=\omega_{1}$, we see that 
\begin{equation}
0\le f\tilde{v}(Bx)-\tilde{v}(x)+\nu x\tilde{v}'(x)\qquad\text{for all }x\ge0.\label{eq:eqholdsfortildev}
\end{equation}
Now we define
\[
v_{\omega_{2}}(x)=\frac{\tilde{v}(Mx)}{\tilde{v}(M)}
\]
for some $M\ge1/B$ to be chosen shortly. Then we have, for $x\ge1/M$
(hence in particular for $x\ge B$) that
\begin{equation}\label{dec1918}
\begin{aligned}
v_{\omega_{2}}(x) & =
\frac{\frac{m}{\omega_{2}}(M^{\omega_{2}}x^{\omega_{2}}-1)+1}
{\frac{m}{\omega_{2}}(M^{\omega_{2}}-1)+1}=
\frac{\frac{m}{\omega_{2}}M^{\omega_{2}}(x^{\omega_{2}}-1)+1-\frac{m}{\omega_2}
+M^{\omega_2}\frac{m}{\omega_2}}
{\frac{m}{\omega_{2}}(M^{\omega_{2}}-1)+1}
\\
&=\frac{M^{\omega_{2}}}{M^{\omega_{2}}-1+\frac{\omega_{2}}{m}}(x^{\omega_{2}}-1)+1.
\end{aligned}
\end{equation}
We note that
\begin{equation}
\omega_2\overset{\cref{eq:omega12ordering}}<\omega_1\overset{\cref{eq:momegacond}}\le m_{\omega_1}\overset{\cref{eq:mchoice}}<m.
\end{equation}
Therefore, defining 
\[
m_{\omega_{2}}\coloneqq\frac{\omega_{2}M^{\omega_{2}}}{M^{\omega_{2}}-1+\frac{\omega_{2}}{m}}
\]
and taking $M$ sufficiently large, we have
    \begin{equation}
\omega_2< m_{\omega_2}<\alpha^{-1}\omega_2.
\end{equation}
Therefore, \cref{eq:momegacond} is satisfied for $\omega=\omega_2$, while~\cref{eq:vomegalargeval}  
holds for $\omega=\omega_{2}$ because of~\cref{dec1918}. The inequality
\cref{eq:eqholdsfortildev} implies that \cref{eq:fvBomega}
holds for $\omega=\omega_{2}$ as well. The other required properties of $v_{\omega_2}$ are clear. Thus $\omega_2\in\clK$, and we have reached a contradiction since we assumed in \cref{eq:omega2def} that $\omega_2\not\in\clK$. Therefore, $(0,1]\setminus\clK$ is empty, and the proof is complete.%
\end{proof}

\section{The majority voting scheme}\label{sec:majvote}
In this section we prove \cref{prop:majorityiskpp}.
First, we compute the derivative of $F$.
\begin{prop}
    We have, for all $u\in [-1,1]$, that
    \begin{equation}\label{eq:Fprimeu}
        F'(u) = \sum_{n\ge 1}2^{1-n}p_n \sum_{k=1}^n \eta_{n,k} (n-k+1)\binom n {k-1}(1-u)^{n-k} (1+u)^{k-1}. %
    \end{equation}
\end{prop}
\begin{proof}
    We recall the integral formula for the binomial CDF (see e.g.~\cite[p.~52]{WB60}):  
  \[
        \Pr_p(X_n\le k) = (n-k)\binom n k \int_0^{1-p} q^{n-k-1}(1-q)^k\,\dif q,
        ~~\hbox{for $0\le k\le n$.}
    \]
    This means that
    \[
        \frac\dif{\dif p}\Pr_p(X_n \ge k)=-\frac\dif{\dif p}\Pr_p(X_n \le k-1) = (n-k+1)\binom n{k-1}(1-p)^{n-k}p^{k-1}.
    \]
    Using this with \cref{eq:Fdef}, we obtain \cref{eq:Fprimeu}.
\end{proof}

\begin{proof}[Proof of \cref{prop:majorityiskpp}]
    First we compute $F'(0)$. From \cref{eq:Fprimeu} we derive
    \begin{equation}\label{eq:Fprime0prelim}
        F'(0) = \sum_{n\ge 1}2^{1-n}p_n \sum_{k=1}^n \eta_{n,k}(n-k+1)\binom n{k-1}.
    \end{equation}
    If $n$ is odd, then the inner sum is
    \begin{equation}\sum_{k=1}^n \eta_{n,k}(n-k+1)\binom n{k-1} = \frac{n+1}2\binom n{(n-1)/2}=\lceil n/2\rceil\binom n{\lfloor n/2\rfloor}.\label{eq:Fprime0odd}\end{equation}
    On the other hand, if $n$ is even, then the inner sum is
    \[\sum_{k=1}^n \eta_{n,k}(n-k+1)\binom n{k-1} = \frac12 (n/2+1)\binom n{n/2-1} +\frac12\cdot \frac n 2\binom n{n/2}.\]
    We note (still assuming that $n$ is even) that
    \begin{equation}
        \binom n {n/2-1} (n/2+1) = \frac{n!}{(n/2-1)!(n/2)!}=\binom{n}{n/2}\frac n2\label{eq:binomident}
    \end{equation}
    to see that 
    \begin{equation}\label{eq:Fprimezeroeven}\sum_{k=1}^n \eta_{n,k}(n-k+1)\binom n{k-1} = \frac n 2\binom n{n/2} = \lceil n/2\rceil \binom n {\lfloor n/2\rfloor},\end{equation}
    which in light of \cref{eq:Fprime0odd} now holds for all $n$.
    Using this in \cref{eq:Fprimeu}, we conclude \cref{eq:Fprime0}.

    Now we show that $F$ is concave on $[0,1]$, which will imply \cref{eq:KPP}. %
    Differentiating \cref{eq:Fprimeu}, we get
    \begin{align*}
        F''(u) &= \sum_{n\ge 1}2^{1-n}p_n \sum_{k=1}^n \eta_{n,k} (n-k+1)\binom n {k-1}\left[(k-1)(1-u)-(n-k) (1+u)\right]\\&\qquad\qquad\qquad\qquad\times(1-u)^{n-k-1}(1+u)^{k-2}. %
    \end{align*}
    When $u\ge 0$, we can write
    \[(k-1)(1-u)-(n-k)(1+u)\le k-1-(n-k)=2k-1-n,\]
    so we obtain, for $u\ge 0$,
    \begin{equation}
        F''(u) \le \sum_{n\ge 1}2^{1-n}p_n \sum_{k=1}^n \eta_{n,k} \binom n {k-1}(n-k+1)(2k-1-n)%
        (1-u)^{n-k-1}(1+u)^{k-2}. %
        \label{eq:fprimeprimebd}
    \end{equation}
    Now, using the choice \cref{eq:majorityetas}, we see that, for odd $n$, we have 
    $\eta_{n,k} = \delta_{k,\frac{n+1}{2}}$, so the $n$th term of the outer sum in \cref{eq:fprimeprimebd} is zero.  On the other hand, for even $n$, the inner sum becomes
    \begin{equation}
        \label{eq:eventerms}
    \begin{aligned}
        \sum_{k=1}^n &\eta_{n,k} \binom n {k-1}(n-k+1)(2k-1-n)
        (1-u)^{n-k-1}(1+u)^{k-2}.\\
                     &=-\frac12 \binom n {n/2-1}(n/2+1)
        (1-u)^{n/2-1}(1+u)^{n/2-2}
                   \\&\qquad+\frac12\binom n {n/2}(n/2)
        (1-u)^{n/2-2}(1+u)^{n/2-1}.
    \end{aligned}
\end{equation}
    We note that
    \[(1-u)^{n/2-2}(1+u)^{n/2-1}=\frac{1+u}{1-u}(1-u)^{n/2-1}(1+u)^{n/2-2} \ge (1-u)^{n/2-1}(1+u)^{n/2-2}\]
    for $u\ge 0$. Using this and \cref{eq:binomident} in \cref{eq:eventerms}, we see that
        \[\sum_{k=1}^n \eta_{n,k} \binom n {k-1}(n-k+1)(2k-1-n)
        (1-u)^{n-k-1}(1+u)^{k-2}\le 0,\]
        now (since we have already shown it for odd $n$) for all $n\ge 1$. Using this in \cref{eq:fprimeprimebd}, this means that $F''(u)\le 0$ for all $u\in [0,1]$, which means that $F$ is concave on $[0,1]$ and hence \cref{eq:KPP} holds.

        Finally, we assume that there is some $n\ge 3$ with $p_n>0$. We have from \cref{eq:Fprimeu} that
        \[F'(1) = \sum_{n\ge 1}p_n \eta_{n,n}n.\]
        For the majority voting scheme, we have \[\eta_{n,n} = \begin{cases}1,&n=1;\\1/2,&n=2;\\0,&n\ge 3,\end{cases}\]
        so
        \[F'(1) = p_1 + p_2<1.\] %
        Since $F(0)=0$ and $F(1)=1$, this and the concavity of $F$ tell us that $F'(0)>1$ and the equation $F(v)=v$ cannot have any solutions in $(0,1)$, so $\Xi_F = 1$.
\end{proof}

\emergencystretch=1em
\printbibliography
\end{document}